%% file: main.tex
\documentclass[twoside]{article}
\usepackage[accepted]{aistats2018}
\usepackage{graphics} 
\usepackage{epsfig} 
\usepackage{mathptmx} 
\usepackage{times} 
\usepackage{amsmath, amsthm, amssymb}
\usepackage{xcolor} 
\usepackage{newtxtext, newtxmath} 
\usepackage[]{hyperref} 
\usepackage{enumitem} 
\usepackage[numbers]{natbib}
\usepackage{subfig}
\usepackage{algorithmicx,algpseudocode,algorithm}

\newtheorem{theorem}{Theorem}
\newtheorem{lemma}[theorem]{Lemma}

\newtheorem{corollary}[theorem]{Corollary}
\newtheorem{definition}[theorem]{Definition}

\theoremstyle{definition}
\newtheorem{experiment}{Experiment} 

\newcommand*{\nolink}[1]{{\protect\NoHyper#1\protect\endNoHyper}}

\newtheorem{innercustomthm}{Theorem}
\newenvironment{customthm}[1]
  {\renewcommand\theinnercustomthm{\nolink{#1}}\innercustomthm}
  {\endinnercustomthm}

\newtheorem{innercustomlemma}{Lemma}
\newenvironment{customlemma}[1]
  {\renewcommand\theinnercustomlemma{\nolink{#1}}\innercustomlemma}
  {\endinnercustomlemma}

\newcommand{\inv}{^{-1}}

\newcommand{\E}{\mathbb{E}}    
\newcommand{\I}{\mathcal{I}} 
\renewcommand{\hat}{\widehat}
\renewcommand{\L}{\mathcal{L}} 
\newcommand{\pr}{\mathbb{P}}   
\newcommand{\R}{\mathbb{R}}    
\renewcommand{\S}{\mathcal{S}}   

\newcommand{\argmin}{\mathop{\arg\!\min}}
\renewcommand{\epsilon}{\varepsilon}

\begin{document}

\twocolumn[

\aistatstitle{Minimax Reconstruction Risk of Convolutional Sparse Dictionary Learning}

\aistatsauthor{ Shashank Singh \And Barnab\'as P\'oczos \And Jian Ma}

\aistatsaddress{\texttt{sss1@cs.cmu.edu}\\Machine Learning Department
\And \texttt{bapoczos@cs.cmu.edu}\\Machine Learning Department\\Carnegie Mellon University
\And \texttt{jianma@cs.cmu.edu}\\Computational Biology Department
}
]

\begin{abstract}
Sparse dictionary learning (SDL) has become a popular method for learning parsimonious representations of data, a fundamental problem in machine learning and signal processing. While most work on SDL assumes a training dataset of independent and identically distributed (IID) samples, a variant known as convolutional sparse dictionary learning (CSDL) relaxes this assumption to allow dependent, non-stationary sequential data sources. Recent work has explored statistical properties of IID SDL; however, the statistical properties of CSDL remain largely unstudied. This paper identifies minimax rates of CSDL in terms of reconstruction risk, providing both lower and upper bounds in a variety of settings. Our results make minimal assumptions, allowing arbitrary dictionaries and showing that CSDL is robust to dependent noise. We compare our results to similar results for IID SDL and verify our theory with synthetic experiments.
\end{abstract}

\section{Introduction}
\label{sec:introduction}

Many problems in machine learning and signal processing can be reduced to, or greatly simplified by, finding a parsimonious representation of a dataset. In recent years, partly inspired by models of visual and auditory processing in the brain~\citep{olshausen97sparseCoding,lewicki00overcompleteRepresentations,smith06efficientAuditory}, the method of \emph{sparse dictionary learning} (SDL; also called \emph{sparse coding}) has become a popular way of learning such a representation, encoded as a sparse linear combination of learned \emph{dictionary elements}, or patterns recurring throughout the data.

SDL has been widely applied to image processing problems such as denoising, demosaicing, and inpainting~\citep{kreutz03dictionary,elad06image,aharon06KSVD,mairal08imageRestoration,peyre09sparseTextures}, separation~\citep{peyre07separation}, compression~\citep{bryt08compression}, object recognition~\citep{kavukcuoglu10fast,thiagarajan2014multiple}, trajectory reconstruction~\citep{zhu15trajectoryReconstruction}, and super-resolution reconstruction~\citep{protter09super,gu15superResolution}. In audio processing, it has been used for structured~\citep{schmidt07wind} and unstructured~\citep{jafari11fastSpeech} denoising, compression~\citep{engan1999method}, speech recognition~\citep{smit09continuousSpeech}, speaker separation~\citep{smaragdis07convolutiveSpeechBases}, and music genre classification~\citep{yeh2012supervisedMusic}. SDL has also been used for both supervised~\citep{mehta2013sparsity,maurer2013sparse} and unsupervised~\citep{argyriou07featureLearning} feature learning in general domains.

The vast majority of SDL literature assumes a training dataset consisting of a large number of independent and identically distributed (IID) samples.
Additionally, the dimension of these samples must often be limited due to computational constraints.
On the other hand, many sources of sequential data, such as images or speech, are neither independent nor identically distributed.
Consequently, most of the above SDL applications rely on first segmenting data into small (potentially overlapping) ``patches'', which are then treated as IID. SDL is then applied to learn a ``local'' dictionary for sparsely representing patches. This approach suffers from two major drawbacks. First, the learned dictionary needs to be quite large, because the model lacks translation-invariance within patches. Second, the model fails to capture dependencies across distances larger than patch sizes, resulting in a less sparse representation. These factors in turn limit computational and statistical performance of SDL.

To overcome this, recent work has explored \emph{convolutional sparse dictionary learning} (CSDL; also called \emph{convolutional sparse coding}), in which a global dictionary is learned directly from a sequential data source, such as a large image or a long, potentially non-stationary, time series~\citep{olshausen02sparseCodesTemporal,smaragdis07convolutiveSpeechBases,smit09continuousSpeech}. In the past few years, CSDL algorithms have demonstrated improved performance on several of the above applications, compared to classical (IID) SDL~\citep{bristow13fast,wohlberg14efficientConvolutional,heide15fastConvolutional,huang15convolutionalDictionary,chun17acceleratingCDL}.

The {\bf main goal} of this work is to understand the statistical properties of CSDL, in terms of \emph{reconstruction} (or \emph{denoising}) error (i.e., the error of reconstructing a data sequence from a learned convolutional sparse dictionary decomposition, a process widely used for denoising and compression~\citep{engan1999method,kreutz03dictionary,bryt08compression,zhu15trajectoryReconstruction,gu15superResolution}). We do this by (a) upper bounding the reconstruction risk of an established estimator~\citep{olshausen02sparseCodesTemporal} and (b) lower bounding, in a minimax sense, the risk of reconstructing a data source constructed sparsely from a convolutional dictionary. The emphasis in this paper is on proving results under minimal assumptions on the data, in the spirit of recent results on ``assumptionless'' consistency of the LASSO~\citep{chatterjee13assumptionlessLASSO,rigollet15HiDStatsNotes}. 
As such, we make no assumptions whatsoever on the dictionary or encoding matrices (such as restricted eigenvalue or isometry properties). 

Compared even to these ``assumptionless'' results, we consider dropping yet another assumption, namely that noise is independently distributed.
In many of the above applications (such as image demosaicing or inpainting, or structured denoising), noise is strongly correlated across the data, and yet dictionary learning approaches nevertheless appear to perform consistent denoising.
To the best of our knowledge, this phenomenon has not been explained theoretically.
One likely reason is that this phenomenon does \emph{not} occur in basis pursuit denoising, LASSO, and many related compressed sensing problems, where consistency is usually not possible under arbitrarily dependent noise.
In the context of SDL, tolerating dependent noise is especially important in the convolutional setting, where coordinates of the data are explicitly modeled as being spatially or temporally related.
However, this phenomenon is also not apparent in the recent work of \citet{papyan17CSDLTheory}, which, to the best of our knowledge, is the only work analyzing theoretical guarantees of CSDL, but considers only deterministic bounded noise and studies recovery of the true dictionary, rather than reconstruction error.
Under sufficient sparsity, our results imply consistency (in reconstruction error) of CSDL, even under arbitrary noise dependence.

{\bf Paper Organization:} Section~\ref{sec:notation} defines notation needed to formalize CSDL and our results. Section~\ref{sec:background} provides background on IID and convolutional SDL. Section~\ref{sec:related_work} reviews related theoretical work. Section~\ref{sec:theory} contains statements of our main theoretical results (proven in the Appendix), which are experimentally validated in Section~\ref{sec:experiments} and discussed further in Section~\ref{sec:discussion}, with suggestions for future work.

\section{Notation}
\label{sec:notation}
Here, we define some notation used throughout the paper.

{\bf Multi-convolution:} For two matrices $R \in \R^{(N - n + 1) \times K}$ and $D \in \R^{n \times K}$ with an equal number of columns, we define the \emph{multi-convolution} operator $\otimes$ by
\[R \otimes D
  := \sum_{k = 1}^K R_k * D_k \in \R^N,\]
$R_k$ and $D_k$ denote the $k^{th}$ columns of $R$ and $D$, respectively, and $*$ denotes the standard discrete convolution operator. In the CSDL setting, multi-convolution (rather than standard matrix multiplication, as in IID SDL) is the process by which data is constructed from the encoding matrix $R$ and the dictionary $D$. We note that, like matrix multiplication, multi-convolution is a bilinear operation.

{\bf Matrix norms:}
For any matrix $A \in \R^{n \times m}$ and $p,q \in [0,\infty]$, the $\L_{p,q}$ norm\footnote{When $\min\{p,q\} < 1$, $\|\cdot\|_{p,q}$ is not a norm (it is not sub-additive). Nevertheless, we will say ``norm'' for simplicity.} of $A$ is
\[\|A\|_{p,q}
  := \left( \sum_{j = 1}^m \left( \sum_{i = 1}^n \left| a_{i,j} \right|^p \right)^{q/p} \right)^{1/q}
  = \left( \sum_{j = 1}^m \|A_j\|_p^q \right)^{1/q}\]
(or the corresponding limit if $p$ or $q$ is $0$ or $\infty$) denotes the $q$-norm of the vector whose entries are $p$-norms of columns of $A$. Note that $\|\cdot\|_{2,2}$ is precisely the Frobenius norm.

{\bf Problem Domain:} For positive integers $N$, $n$, and $K$, we use
\[\S := \left\{ (R, D) \in \R^{(N - n + 1) \times K} \times \R^{n \times K} :
              \|D\|_{2,\infty} \leq 1
         \right\}\]
to denote the domain of the dictionary learning problem, (i.e., $(R, D) \in \S$, as described in the next section), and, for any $\lambda \geq 0$, we further use
\[\S_\lambda := \left\{ (R, D) \in \S : \|R\|_{1,1} \leq \lambda \right\}\]
to denote the $\L_{1,1}$-constrained version of this domain. Note that both $\S$ and $\S_\lambda$ are convex sets.

\section{Background: IID and Convolutional Sparse Dictionary Learning}
\label{sec:background}
We now review the standard formulations of IID and convolutional sparse dictionary learning.

\subsection{IID Sparse Dictionary Learning}
\label{subsec:IID_dictionary_learning}

The IID SDL problem considers a dataset $Y \in \R^{N \times d}$ of $N$ IID samples with values in $\R^d$. The goal is to find an approximate decomposition $Y \approx RD$, where $R \in \R^{N \times K}$ is a sparse encoding matrix and $D \in \R^{K \times d}$ is a dictionary of $K$ patterns.
A frequentist starting point for IID SDL is the \emph{linear generative model}~\citep{olshausen97sparseCoding} (LGM), which supposes there exist $R$ and $D$ as above such that
\begin{equation}
Y = X + \epsilon \in \R^{N \times d},
\label{eq:LGM}
\end{equation}
where $X = RD$ and $\epsilon$ is a random noise matrix with independent rows. Under the assumption that $R$ is sparse, a natural approach to estimating the model parameters $R$ and $D$ is to solve the $\L_{1,1}$-constrained optimization problem
\begin{align}
\label{opt:classical_dict_learn}
\left( \hat R_\lambda, \hat D_\lambda \right)
&  = \argmin_{(R,D)} \|Y - RD\|_{2,2}^2 \\
\notag
  \quad \text{ subject to } \quad
& \|R\|_{1,1} \leq \lambda \quad \text{ and } \quad \|D\|_{2, \infty} \leq 1,
\end{align}
where the minimization is over all $R \in [0,\infty)^{N \times K}$ and $D \in \R^{K \times d}$. Here, $\lambda \geq 0$ is a tuning parameter controlling the sparsity of the estimate $\hat R_\lambda$; the $\L_{1,1}$ sparsity constraint can be equivalently expressed as a penalty of $\lambda'\|R\|_{1,1}$ (where $\lambda' \neq \lambda$) on the objective.
Inspired by non-negative matrix factorization~\citep{lee01NMFalgorithms}, $R$ is sometimes constrained to be non-negative to promote interpretability -- it is often more natural to consider a negative multiple of a feature to be a different feature altogether -- but this does not significantly affect theoretical analysis of the problem.
The constraint $\|D\|_{2, \infty} \leq 1$ normalizes the size of the dictionary entries; without this, $\|R\|_{1,1}$ could become arbitrarily small without changing $RD$, by scaling $D$ correspondingly.

Since matrix multiplication is bilinear, the optimization problem~\eqref{opt:classical_dict_learn} is not jointly convex in $R$ and $D$, but it is \emph{biconvex}, i.e., convex in $R$ when $D$ is fixed and convex in $D$ when $R$ is fixed.
This enables, in practice, a number of iterative optimization algorithms, typically based on alternating minimization, i.e., alternating between minimizing~\eqref{opt:classical_dict_learn} in $R$ and in $D$.
Interestingly, recent work \citep{sun15completeRecovery,sun15nonconvex} has shown that, despite being non-convex, the SDL problem is often well-behaved such that standard iterative optimization algorithms can sometimes provably converge to global optima, even without multiple random restarts.

\subsection{Convolutional Sparse Dictionary Learning}
\label{subsec:convolutional_dictionary_learning}

The CSDL problem considers a single data vector $Y \in \R^N$, where $N$ is assumed to be very large.
For example, $Y$ might be a speech~\citep{smaragdis07convolutiveSpeechBases,smit09continuousSpeech} or music~\citep{yeh2012supervisedMusic} signal over time, or functional genomic data~\citep{alipanahi2015predicting} or sequence data~\citep{zhou15deepsea,singh16SPEID} over the length of the genome. Simple extensions can consider, for example, large, high-resolution images by letting $Y \in \R^{N_1 \times N_2}$ be $2$-dimensional~\citep{mairal09online}.
As discussed in Section~\ref{sec:discussion}, CSDL can also generalize in multiple ways to multichannel data $Y \in \R^{N \times d}$ (e.g., to handle multiple parallel audio streams or functional genomic signals, or color images). To keep our main results simple, this paper only considers a single-dimensional ($Y \in \R^N$), single-channel ($d = 1$) signal, which already presents interesting questions, (whereas IID SDL with $d = 1$ degenerates to estimating a sparse sequence).

The goal here is to find an approximate decomposition $Y \approx R \otimes D$, where $R \in \R^{(N - n + 1) \times K}$ is a sparse encoding matrix and $D \in \R^{n \times K}$ is a dictionary of $K$ patterns.\footnote{This choice of notation implies some coupling between parameters $n$ and $N$ (namely $n \leq N$), but we usually have $n \ll N$, and our discussion involves the length of $R \otimes D$ (the sample size) more than the length of $R$, so it is convenient to use $N$ for the former.} CSDL can also be studied in a frequentist model, the \emph{temporal linear generative model} (TLGM)~\citep{olshausen02sparseCodesTemporal}, which supposes that there exist $R$ and $D$ as above such that
\begin{equation}
Y = X + \epsilon \in \R^N,
\label{eq:TLGM_model}
\end{equation}
where $X = R \otimes D$, and, again, $\epsilon$ is random noise, though, in this setting, it may make less sense to assume that the rows of $\epsilon$ are independent.
Under the assumption that $R$ is sparse, a natural approach to estimating the model parameters $R$ and $D$ is again to solve an $\L_{1,1}$-constrained optimization problem, this time
\begin{equation}
(\hat R_\lambda, \hat D_\lambda)
  := \argmin_{(R, D) \in \S_\lambda} \left\| Y - R \otimes D \right\|_2^2
\label{opt:constrained_CSDL}
\end{equation}
where the minimization is over all $R \in [0, \infty)^{(N - n + 1) \times K}$ and $D \in \R^{n \times K}$. Since multi-convolution is bilinear, the optimization problem~\eqref{opt:constrained_CSDL} is again biconvex, and can be approached by alternating minimization. As with the IID case, the constrained optimization problem~\eqref{opt:constrained_CSDL} can equivalently be expressed as a penalized problem, specifically,
\begin{equation}
\left( \hat R_{\lambda'}, \hat D_{\lambda'} \right)
  = \argmin_{(R,D) \in \S} \left\| Y - R \otimes D \right\|_2^2 + \lambda'\|R\|_{1,1}
\label{opt:penalized_CSDL}
\end{equation}
(where, again, $\lambda \neq \lambda'$). For the remainder of this paper, we will discuss the constrained problem~\eqref{opt:constrained_CSDL}, but we show in the Appendix that equivalent results hold for the penalized problem~\eqref{opt:penalized_CSDL}.

To summarize, the key differences between the IID and convolutional SDL problems setups are:

\begin{enumerate}[wide,topsep=0pt,partopsep=1pt,parsep=1pt]
\setlength{\itemsep}{0em}
\item
In CSDL, we seek a decomposition $Y \approx R \otimes D$, whereas, in IID SDL, we seek a decomposition $Y \approx RD$. Unlike matrix multiplication, by which each row of $R$ corresponds to a single row of $X$, multi-convolution allows each row of $R$ to contribute to up to $n$ consecutive rows of $X$, modeling, for example, temporally or spatially related features.
\item
In CSDL, the noise $\epsilon$ may have arbitrary dependencies, whereas, in IID SDL, it typically has independent rows.
\item
CSDL involves an additional (potentially unknown) parameter $n$ controlling the length of the dictionary entries,\footnote{In fact, $n$ can be distinct for each of the $K$ features, suggesting a natural approach to learning \emph{multi-scale} convolutional dictionaries, which are useful in many contexts. We leave this avenue for future work.} whereas, in IID SDL, $n = d$ is known.
\end{enumerate}

\section{Related Work}
\label{sec:related_work}
There has been some work theoretically analyzing the non-convex optimization problem~\eqref{opt:classical_dict_learn} in terms of which IID SDL is typically cast~\citep{mairal09online,sun15nonconvex,sun15completeRecovery}, with a consensus that despite being non-convex, this problem is often efficiently solvable in practice. Our work focuses on the statistical aspects of CSDL, and we assume, for simplicity, that the global optimum of the CSDL optimization problem~\eqref{opt:constrained_CSDL} can be computed to high accuracy; more work is needed to link the parameters of the CSDL problem to efficient and accurate computability of this optimum.

There has been some work analyzing statistical properties of IID SDL. For some algorithms, upper bounds have been shown on the risk (in Frobenius norm, up to permutation of the dictionary elements) of estimating the true dictionary $D$~\citep{agarwal14learning,arora14newAlgorithms}.
\citet{vainsencher11sampleComplexity} studied generalizability of dictionary learning in terms of representation error of a learned dictionary on an independent test set from the same distribution as the training set.
Recently, \citet{jung14performanceLimits,jung16minimax} proved the first minimax lower bounds for IID SDL, in several settings, including a general dense model, a sparse model, and a sparse Gaussian model.

The work most closely related to ours is that of \citet{papyan17CSDLTheory}, who recently began studying the numerical properties of the CSDL. Importantly, they show that, although CSDL~\eqref{opt:constrained_CSDL} can be expressed as a constrained version of IID SDL~\eqref{opt:classical_dict_learn}, a novel, direct analysis of CSDL, in terms of more refined problem parameters leads to stronger guarantees than does applying analysis from IID SDL. Their results are complementary to ours, for several reasons:
\begin{enumerate}[wide,topsep=0pt,partopsep=1pt,parsep=1pt]
\setlength{\itemsep}{0pt}
\item
We study error of reconstructing $X = R \otimes D$, whereas they studied recovery of the dictionary $D$, which requires strong assumptions on $D$ (see below).
\item
They consider worst-case deterministic noise of bounded $\L_2$ norm, whereas we consider random noise under several different statistical assumptions.
Correspondingly, we give (tighter) bounds on expected, rather than worst-case, error.
\item
They study the $\L_{0,0}$-``norm'' version of the problem, while we study the relaxed $\L_{1,1}$-norm version. By comparison to analysis for best subset selection and the LASSO in linear regression, we might expect solutions to the $\L_{0,0}$ problem to have superior statistical performance in terms of reconstruction error, and that stronger assumptions on the dictionary $D$ may be necessary to recover $D$ via the $\L_{1,1}$ approach than via the $\L_{0,0}$ approach. Conversely, the $\L_{0,0}$ problem is NP-hard, whereas the $\L_{1,1}$ problem can be solved via standard optimization approaches, and is hence used in practice.
\footnote{Recent algorithmic advances in mixed integer optimization have rendered best subset selection computable for moderately large problem sizes~\citep{bertsimas16best}. It remains to be seen how effectively these methods can be leveraged for SDL.}
\end{enumerate}

Notably, these previous results all require strong restrictions on the structure of the dictionary. These restrictions have been stated in several forms, from incoherence assumptions~\citep{jenatton12local,agarwal14learning} and restricted isometry conditions~\citep{jung14performanceLimits,jung16minimax,papyan17CSDLTheory} to bounds on the Babel function~\citep{tropp04greed,vainsencher11sampleComplexity} of the dictionary, but all essentially boil down to requiring that the dictionary elements are not too correlated. As discussed in \citet{papyan17CSDLTheory}, the assumptions needed in the convolutional case are even stronger than in the IID case: no \emph{translations} of the dictionary elements can be too correlated. Furthermore, these conditions are not verifiable in practice. A notable feature of our upper bounds is that they make no assumptions whatsoever on the dictionary $D$; this is possible because our bounds apply to reconstruction error, rather than to the error of learning the dictionary itself. As noted above, this can be compared to the fact that, in sparse linear regression, bounds on prediction error can be derived with essentially no assumptions on the covariates~\citep{chatterjee13assumptionlessLASSO}, whereas much stronger assumptions, to the effect that the covariates are not too strongly correlated, are needed to derive bounds for estimating the linear regression coefficients.

Finally, as noted earlier, we make minimal assumptions on the structure of the noise; in particular, though we require the noise to have light tails (in either a sub-Gaussian or finite-moment sense), we allow arbitrary dependence across the data sequence. This is important because, in many applications of CSDL, errors are likely to be correlated with those in nearby portions of the sequence. In the vast compressed sensing literature, there is relatively little work under these very general conditions, likely because most problems in this area, such as basis pursuit denoising or the LASSO, are clearly not consistently solvable under such general conditions. All the previously mentioned works on guarantees for dictionary learning \citep{vainsencher11sampleComplexity,agarwal14learning,arora14newAlgorithms,jung14performanceLimits,jung16minimax} also assume no or independent noise.

\section{Theoretical Results}
\label{sec:theory}

We now present our theoretical results on the minimax average $\L_2$-risk of reconstructing $X = R \otimes D$ from $Y$, i.e.,
\begin{equation}
M(N, n, \lambda, \sigma)
  := \inf_{\hat X} \sup_{(R, D) \in \S_\lambda} \frac{1}{N} \E \left[ \left\| \hat X - X \right\|_2^2 \right],
\label{exp:minimax_risk}
\end{equation}
where the infimum is taken over all estimators $\hat X$ of $X$ (i.e., all functions of the observation $Y$). The quantity~\eqref{exp:minimax_risk} characterizes the worst-case mean squared error of the average coordinate of $\hat X$, for the best possible estimator $\hat X$. Since it bounds within-sample reconstruction error, these results are primarily relevant for compression and denoising applications, rather than for learning an interpretable dictionary.

\subsection{Upper Bounds for Constrained CSDL}
\label{subsec:upper_bounds}

In this section, we present our upper bounds on the reconstruction risk of the $\L_{1,1}$-constrained CSDL estimator \eqref{opt:constrained_CSDL}, thereby upper bounding the minimax risk~\eqref{exp:minimax_risk}. We begin by noting a simple oracle bound, which serves as the starting point for our remaining upper bound results.

\begin{lemma}
Let $Y = X + \epsilon \in \R^N$.
Then, for any $\left( R, D \right) \in \S_\lambda$,
\begin{align}
\notag
& \| X - \hat R_\lambda \otimes \hat D_\lambda \|_2^2 \\
& \leq \| X -  R \otimes D \|_2^2 + 2\langle \epsilon, \hat R_\lambda \otimes \hat D_\lambda - R \otimes D \rangle.
\label{ineq:oracle_inequality}
\end{align}
\label{lemma:oracle_inequality}
\end{lemma}

This result decomposes the error of constrained CSDL into error due to model misspecification:
\begin{equation}
\| X -  R \otimes D \|_2^2
\label{exp:model_misspec_error}
\end{equation}
and statistical error:
\begin{equation}
2\langle \epsilon, \hat R_\lambda \otimes \hat D_\lambda - R \otimes D \rangle.
\label{exp:stochastic_error}
\end{equation}

For simplicity, our remaining results assume the TLGM~\eqref{eq:TLGM_model} holds, so that~\eqref{exp:model_misspec_error} is $0$. We therefore focus on bounding the statistical error~\eqref{exp:stochastic_error}, uniformly over $(R,D) \in \S_\lambda$. However, the reader should keep in mind that our upper bounds hold, with an additional $\inf_{(R,D) \in \S_\lambda} \| X -  R \otimes D \|_2^2$ term, without the TLGM. Equivalently, our results can be considered bounds on excess risk relative to the optimal sparse convolutional approximation of $X$. In particular, this suggests robustness of constrained CSDL to model misspecification.

Our main upper bounds apply under sub-Gaussian noise assumptions. We distinguish two notions of multivariate sub-Gaussianity, defined as follows:

\begin{definition}[Componentwise Sub-Gaussianity]
An $\R^N$-valued random variable $\epsilon$ is said to be componentwise sub-Gaussian with constant $\sigma > 0$ if
\[\sup_{i \in \{1,...,N\}} \E \left[ e^{t \epsilon_i} \right]
  \leq e^{t^2 \sigma^2/2},
  \quad \text{ for all } t \in \R.\]
\end{definition}

\begin{definition}[Joint Sub-Gaussianity]
An $\R^N$-valued random variable $\epsilon$ is said to be jointly sub-Gaussian with constant $\sigma > 0$ if
\[\E \left[ e^{\langle t, \epsilon \rangle} \right]
  \leq e^{\|t\|_2^2 \sigma^2/2},
  \quad \text{ for all } t \in \R^N.\]
\end{definition}

Componentwise sub-Gaussianity is a much weaker condition than joint sub-Gaussianity, and, in real data, it is also often more intuitive or measurable.
However, the two definitions coincide when the components of $\epsilon$ are independent, since the moment generating function $\E \left[ e^{\langle t, \epsilon \rangle} \right]$ factors, and, for this reason, joint sub-Gaussianity is commonly assumed in high-dimensional statistical problems~\citep{ledoux13BanachSpaces,boucheron13concentration,rigollet15HiDStatsNotes}.
As we will show, these two conditions lead to different minimax error rates.
For sake of generality, we avoid making any independence assumptions, but our results for jointly sub-Gaussian noise can be thought of as equivalent to results for componentwise sub-Gaussian noise, under the additional assumption that the noise has independent components.

Consider first the case of componentwise sub-Gaussian noise. Then, constrained CSDL satisfies the following:
\begin{theorem}[Upper Bound for Componentwise Sub-Gaussian Noise]
Assume the TLGM holds, suppose the noise $\epsilon$ is componentwise sub-Gaussian with constant $\sigma$, and let the constrained CSDL tuning parameter $\lambda$ satisfy $\lambda \geq \|R\|_{1,1}$. Then, the reconstruction estimate $\hat X_\lambda = \hat R_\lambda \otimes \hat D_\lambda$ satisfies
\begin{equation}
\frac{1}{N} \E \left[ \|\hat X_\lambda - X\|_2^2 \right]
  \leq \frac{4 \lambda \sigma \sqrt{2 n \log(2N)}}{N}.
\label{ineq:dependent_bound}
\end{equation}
\label{thm:dependent_bound}
\end{theorem}

Consider now the case of jointly sub-Gaussian noise. Then, an appropriately tuned constrained CSDL estimate satisfies the following tighter bound:

\begin{theorem}[Upper Bound for Jointly Sub-Gaussian Noise]
Assume the TLGM holds, suppose the noise $\epsilon$ is jointly sub-Gaussian with constant $\sigma$, and let the constrained CSDL tuning parameter $\lambda$ satisfy $\lambda \geq \|R\|_{1,1}$. Then, the reconstruction estimate $\hat X_\lambda = \hat R_\lambda \otimes \hat D_\lambda$ satisfies
\begin{equation}
\frac{1}{N} \E \left[ \|\hat X_\lambda - X\|_2^2 \right]
  \leq \frac{4 \lambda \sigma \sqrt{2 \log(2(N - n + 1))}}{N}.
\label{ineq:independent_bound}
\end{equation}
\label{thm:independent_bound}
\end{theorem}

The main difference between the bounds~\eqref{ineq:dependent_bound} and~\eqref{ineq:independent_bound} is the presence of a $\sqrt{n}$ factor in the former. In Section~\ref{subsec:lower_bounds}, we will show that both upper bounds are essentially tight, and the minimax rates under these two noise conditions are indeed separated by a factor of $\sqrt{n}$. We also empirically verify this phenomenon in Section~\ref{sec:experiments}, and, in Section~\ref{sec:discussion}, we provide some intuition for why this occurs. Also note that, in the Appendix, we provide related results bounding the error of the penalized form~\eqref{opt:penalized_CSDL} of CSDL and also bounding error under weaker finite-moment noise conditions.

\subsection{Lower Bounds}
\label{subsec:lower_bounds}

We now present minimax lower bounds showing that the upper bound rates in Theorems~\ref{thm:dependent_bound} and \ref{thm:independent_bound} are essentially tight in terms of the sparsity $\lambda$, noise level $\sigma$, sequence length $N$, and dictionary length $n$.

First consider the componentwise sub-Gaussian case, analogous to that considered in Theorem~\ref{thm:dependent_bound}:
\begin{theorem}[Lower Bound for Componentwise sub-Gaussian Noise]
Assume the TLGM holds. Then, there exists a (Gaussian) noise pattern $\epsilon$ that is componentwise sub-Gaussian with constant $\sigma$ such that the following lower bound on the minimax average $\L_2$ reconstruction risk holds:
\begin{equation}
M(\lambda, N, n, \sigma) \geq \frac{\lambda}{8N} \min \left\{
         \lambda,
         \sigma \sqrt{n \log(N - n + 1)}
       \right\}
\label{ineq:componentwise_lower_bound}
\end{equation}
\label{thm:componentwise_lower_bound}
\end{theorem}
The $\min$ here reflects the fact that, in the extremely sparse or noisy regime $\lambda \leq \sigma \sqrt{n \log(N - n + 1)}$, the trivial estimator $\hat \theta = 0$ becomes optimal, with average $\L_2$-risk at most $\lambda^2/N$.
Except in this extremely sparse/noisy case, this minimax bound essentially matches the upper bound provided by Theorem~\ref{thm:dependent_bound}.

Rather than directly utilizing information theoretic bounds such as Fano's inequality, our lower bounds are based on reducing the classical $\L_1$-constrained Gaussian sequence estimation problem (see, e.g., Section 2.3 of \citet{rigollet15HiDStatsNotes}) to CSDL (i.e., showing that an estimator for the prescribed CSDL problem can be used to construct a estimator for the mean of an $\L_1$-constrained Gaussian sequence such that the error of the Gaussian sequence estimator is bounded in terms of that of the CSDL estimator). Standard lower bounds for the $\L_1$-constrained Gaussian sequence problem (e.g., Corollary 5.16 of \citet{rigollet15HiDStatsNotes}) then directly imply a lower bound for the CSDL estimator. Again, detailed proofs are provided in the Appendix.

Now consider the jointly sub-Gaussian case, analogous to that considered in Theorem~\ref{thm:independent_bound}:
\begin{theorem}
{\it (Lower Bound for Jointly sub-Gaussian Noise):} Assume the TLGM holds, and suppose that $\epsilon \sim \mathcal{N}(0, \sigma^2 I_N)$, so that $\epsilon$ is jointly sub-Gaussian with constant $\sigma$. Then, the following lower bound on the minimax average $\L_2$ reconstruction risk holds:
\begin{equation}
M(\lambda, N, n, \sigma) \geq \frac{\lambda}{8N} \min \left\{
         \lambda,
         \sigma \sqrt{\log(N - n + 1)}
       \right\}
\label{ineq:joint_lower_bound}
\end{equation}
\label{thm:joint_lower_bound}
\end{theorem}
Again, the $\min$ here reflects the fact that, in the extremely sparse or noisy regime $\lambda \leq \sigma \sqrt{\log(N - n + 1)}$, the trivial estimator $\hat \theta = 0$ becomes optimal.
Except in this extreme case, this minimax bound essentially matches the upper bound provided by Theorem~\ref{thm:independent_bound}.

\subsection{Comparison to IID SDL}
\label{sec:IID_SDL_comparison}
As mentioned previously, the IID SDL algorithm~\eqref{opt:classical_dict_learn} has historically been applied to problems with spatial or temporal structure better modeled by a TLGM~\eqref{eq:TLGM_model} than by an LGM~\eqref{eq:LGM}, by means of partitioning the data into patches.
In this section, we consider how the statistical performance of IID SDL compares with that of CSDL, under the TLGM. For simplicity, we consider the case of componentwise sub-Gaussian noise; conclusions under joint sub-Gaussianity are similar.
For clarity, in this section, notation used according to the LGM/IID SDL setting will be denoted with the prime mark $'$, while other quantities should be interpreted as in the TLGM/CSDL setting.\footnote{In this section, $\lambda'$ should not be confused with the tuning parameter of penalized CSDL, also denoted $\lambda'$.}

The next result is an analogue of Theorem~\ref{thm:dependent_bound} for IID SDL under the LGM; the proof is similar, with Young's inequality for convolutions replaced by a simple linear bound.
\begin{theorem}[Upper Bound for IID SDL]
Assume the LGM holds, suppose the noise $\epsilon$ is componentwise sub-Gaussian with constant $\sigma$ (i.e., for each dimension $j \in [d']$, $\epsilon_j \in \R^{N'}$ is componentwise sub-Gaussian with constant $\sigma$), and let the constrained IID SDL parameter $\lambda'$ satisfy $\lambda' \geq \|R'\|_{1,1}$. Then, the reconstruction estimate $\hat X_{\lambda'}' = \hat R_{\lambda'}' \hat D_{\lambda'}'$ satisfies
\begin{equation}
\frac{1}{N'd'} \|X - \hat X_{\lambda'}'\|_{2,2}^2
  \leq \frac{4 \lambda' \sigma \sqrt{2d' \log(2N'd)}}{N'd'}
\label{ineq:IID_SDL_bound}
\end{equation}
\label{thm:IID_SDL_bound}
\end{theorem}
The natural conversion of parameters from the TLGM to the LGM sets $d' = n$ and $N' = N/n$. At first glance, plugging these into~\eqref{ineq:IID_SDL_bound} appears to give the same rate as~\eqref{ineq:dependent_bound}. The catch is the condition that $\lambda' \geq \|R'\|_{1,1}$. When converting data from the TLGM to the format of the LGM, the sparsity level can grow by a factor of up to $n$; that is, in the worst case, we can have $\|R'\|_{1,1} = n\|R\|_{1,1}$. Therefore, the bound can increase by a factor of $n$, relative to the bound for CSDL. From a statistical perspective, this may explain the superior performance of CSDL over IID SDL, when data is better modeled by the TLGM than by the LGM,  especially when the patterns being modeled in the data are relatively large.

\section{Empirical Results}
\label{sec:experiments}
In this section, we present numerical experiments on synthetic data, with the goal of verifying the convergence rates derived in the previous section. MATLAB code for reproducing these experiments and figures is available at \url{https://github.com/sss1/convolutional-dictionary}. Since the focus of this paper is on the statistical, rather than algorithmic, properties of CSDL, we assume the estimator $(\hat R_\lambda, \hat D_\lambda)$ defined by the optimization problem~\eqref{opt:constrained_CSDL} can be computed to high precision using a simple alternating projected gradient descent algorithm, the details of which are provided in the appendix. We then use the reconstruction estimate $\hat X_\lambda := \hat R_\lambda \otimes \hat D_\lambda$ to estimate $X = R \otimes D$. All results presented are averaged over $1000$ IID trials, between which $R$ and $S$ were regenerated randomly.\footnote{We initially plotted $95\%$ confidence intervals for each point based on asymptotic normality of the empirical $\L_2$ error.
Since intervals were typically smaller than markers sizes, we removed them to avoid clutter.} 
In each trial, the $K$ dictionary elements (columns of $D$) are sampled independently and uniformly from the $\L_2$ unit sphere in $\R^n$.
$R$ is generated by initializing $R = 0 \in \R^{(N - n + 1) \times K}$ and then repeatedly adding $1$ to uniformly random coordinates of $R$ until the desired value of $\|R\|_{1,1}$ is achieved. Further details of the experiment setup and implementation are given in the Appendix.

{\it Comparisons:}
We compare the error of the optimal CSDL estimator $\hat X_\lambda$ to the theoretical upper bounds (Inequalities~\eqref{ineq:dependent_bound} and \eqref{ineq:independent_bound}) and lower bounds (Inequalities~\eqref{ineq:joint_lower_bound} and \eqref{ineq:componentwise_lower_bound}), as well as the trivial estimators $\hat X_0 = 0$ and $\hat X_\infty = Y$.

\begin{experiment}
\label{experiment:1}
Our first experiment studies the relationship between the length $N$ of the sequence and the true $\L_{1,1}$-sparsity $\|R\|_{1,1}$ of the data.
Figure~\ref{fig:exp1} shows error as a function of $N$ for logarithmically spaced values between $10^2$ and $10^4$, with $\|R\|_{1,1}$ scaling as constant $\|R\|_{1,1} = 5$, square-root $\|R\|_{1,1} = \left\lfloor \sqrt{N} \right\rfloor$, and linear $\|R\|_{1,1} = \lfloor N/10 \rfloor$ functions of $N$.
The results are consistent with the main predictions of Theorems~\ref{thm:independent_bound} and \ref{thm:joint_lower_bound} for jointly sub-Gaussian noise, namely that the error of the CSDL estimator using the optimal tuning parameter $\lambda = \|R\|_{1,1}$ lies between the lower and upper bounds, and converges at a rate of order $\|R\|_{1,1}/N$, up to $\log$ factors).
As a result, the estimator is inconsistent when $\|R\|_{1,1}$ grows linearly with $N$, in which case there is no benefit to applying CSDL to denoise the sequence over using the original sequence $\hat X_\infty = X$, even though the latter is never consistent.\footnote{Even if $\|R\|_{1,1}$ grows linearly with $N$, as long as $\|R\|_{1,1}/N$ is small, CSDL may still be useful for compression, if a constant-factor loss is acceptable.}
On the other hand, if $\|R\|_{1,1}$ scales sub-linearly with $N$, CSDL is consistent and outperforms both trivial estimators (although, of course, the trivial estimator $\hat X_0 = 0$ is also consistent in this setting).
In the appendix, we present an analogous experiment in the heavy-tailed case, where the noise $\epsilon$ has only $2$ finite moments.
\begin{figure}
\centering
\includegraphics[width=\linewidth,trim={7mm 0 17mm 0},clip]{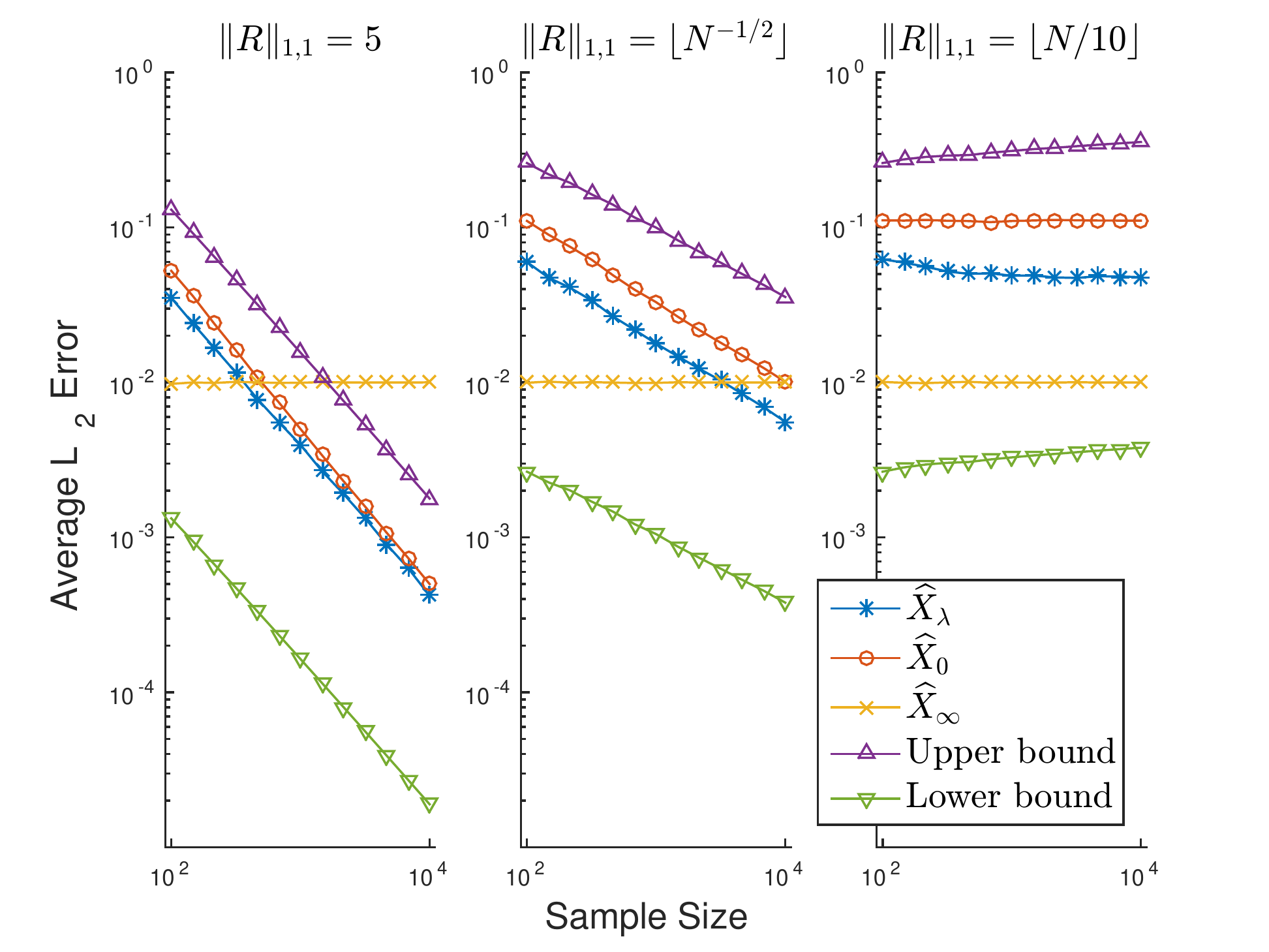}
\caption{Experiment 1: Average $\L_2$-error as a function of sequence length $N$, with sparsity scaling as $\|R\|_{1,1} = 5$ (first panel), $\|R\|_{1,1} = \left\lfloor \sqrt{N} \right\rfloor$ (second panel), and $\|R\|_{1,1} = \lfloor N/10 \rfloor$ (third panel).}
\label{fig:exp1}
\end{figure}
\end{experiment}

\begin{experiment}
\label{experiment:2}
\begin{figure*}
\centering
\subfloat[]{\includegraphics[width=0.5\linewidth,trim={5mm 0 15mm 0},clip]{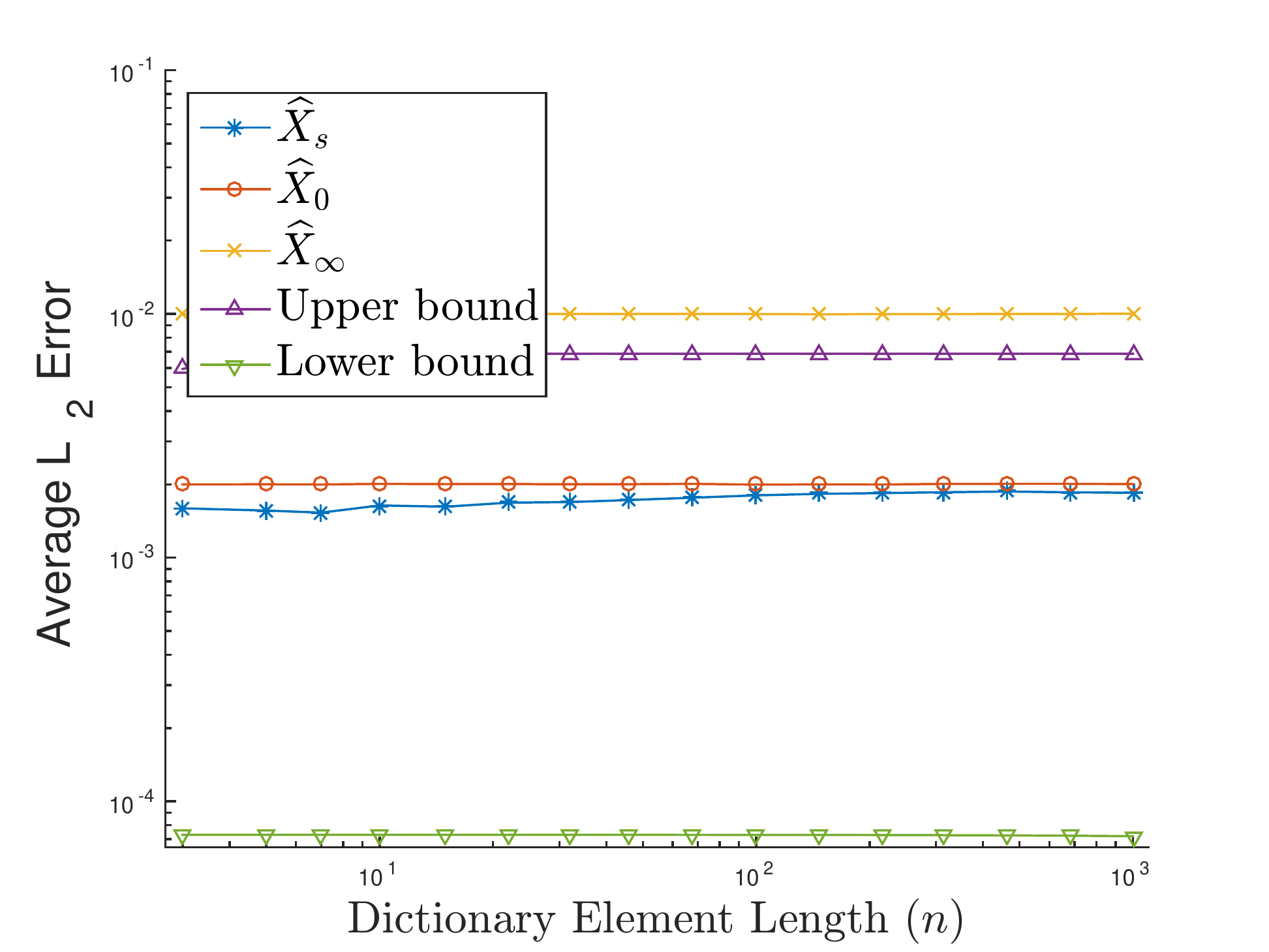}\label{subfig:exp2a}}
\subfloat[]{\includegraphics[width=0.5\linewidth,trim={5mm 0 15mm 0},clip]{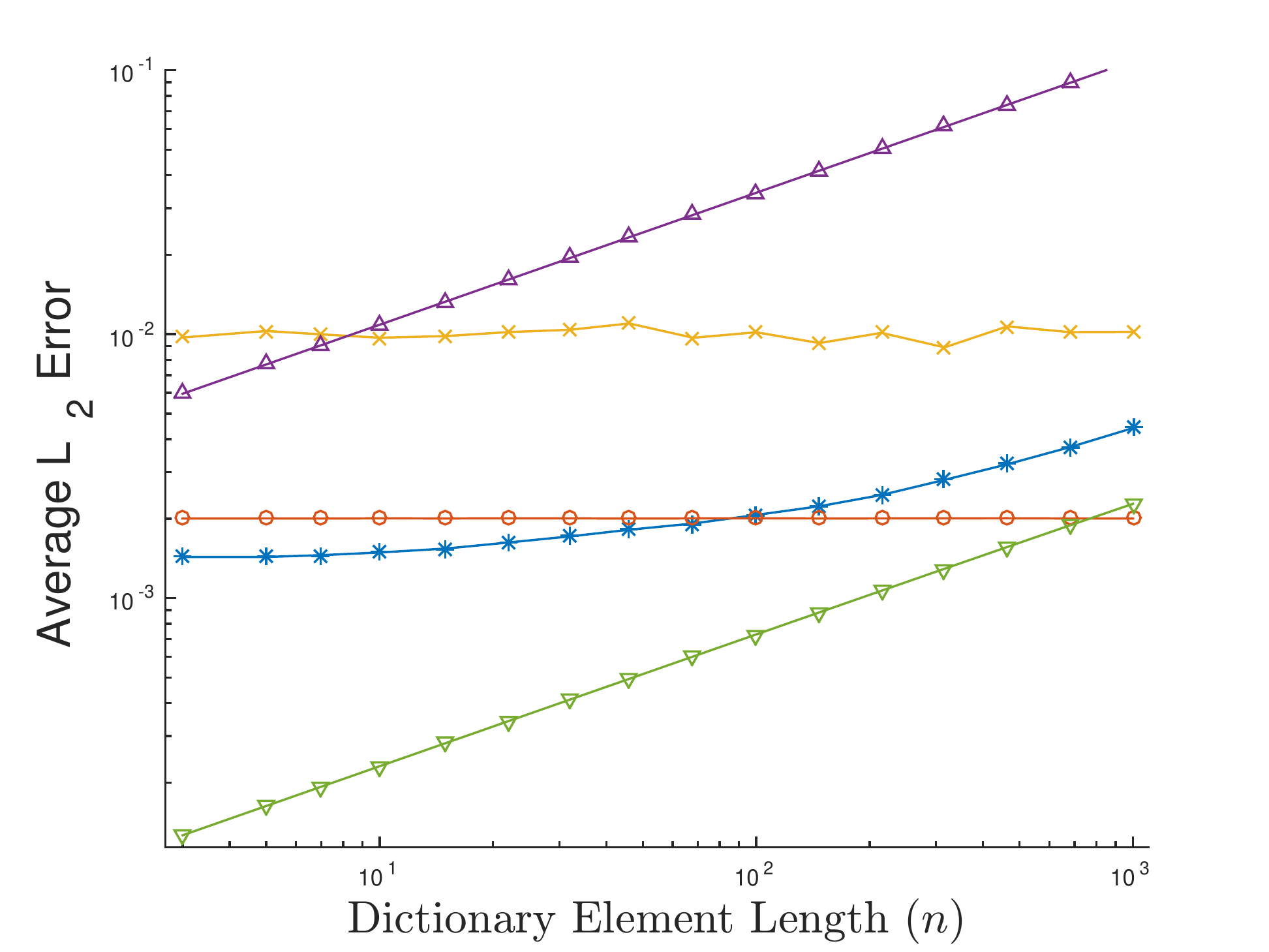}\label{subfig:exp2b}}
  \caption{Experiment~\nolink{\ref{experiment:2}}: Average $\L_2$-error as a function of dictionary element length $n$, when \ref{subfig:exp2a} noise is independent across the input and \ref{subfig:exp2b} noise is perfectly correlated across the input.}
\label{fig:exp2}
\end{figure*}

Our second experiment studies the dependence of the error on the length $n$ of the dictionary elements, and how this varies with the dependence structure of the noise $\epsilon$. Specifically, we considered two ways of generating the noise $\epsilon$: (1) IID Gaussian entries (as in other experiments), and (2) perfectly correlated Gaussian entries (i.e., a single Gaussian sample was drawn and added to every entry of $R \otimes D$. In the former case, Theorems~\ref{thm:independent_bound} and \ref{thm:joint_lower_bound} suggest a convergence rate independent of $n$, whereas, in the latter case, Theorems~\ref{thm:dependent_bound} and \ref{thm:componentwise_lower_bound} suggest a rate scaling as quickly as $\sqrt{n}$. To allow a larger range of values for $n$, we fixed a larger value of $N = 5000$ for this experiment. Figure~\ref{fig:exp2} shows error as a function of $n$ for logarithmically spaced values between $10^{1/2}$ and $10^3$, in the cases of independent noise and perfectly correlated noise. As predicted, the error of the CSDL estimator using the optimal tuning parameter $\lambda = \|R\|_{1,1}$ lies between the lower and upper bounds, and exhibits worse scaling in the case of dependent noise.
\end{experiment}

\section{Discussion}
\label{sec:discussion}
Theorems~\ref{thm:dependent_bound}, \ref{thm:independent_bound}, \ref{thm:componentwise_lower_bound}, and \ref{thm:joint_lower_bound} together have several interesting consequences. Firstly, in the fairly broad setting of componentwise sub-Gaussian noise, for fixed $K$, the minimax average $\L_2$ risk for CSDL reconstruction in the case of sub-Gaussian noise is, up to constant factors, of order
\[M(\lambda, N, n, \sigma)
        \asymp \frac{\lambda}{N} \min \left\{
          \lambda,
          \sigma \sqrt{n \log N}
        \right\}.\]
Under a stronger assumption of joint sub-Gaussianity (e.g., independent noise), the minimax risk becomes
\[M(\lambda, N, n, \sigma)
  \asymp \frac{\lambda}{N} \min \left\{
           \lambda,
           \sigma \sqrt{\log N}
         \right\}.\]
In retrospect, the presence of an additional $\sqrt{n}$ factor in error under componentwise sub-Gaussianity is quite intuitive.
Dictionary learning consists, approximately, of performing compressed sensing in an (unknown) basis of greatest sparsity.
Componentwise sub-Gaussianity bounds the noise level in the original (data) basis, whereas joint sub-Gaussianity bounds the noise level in \emph{all} bases.
Hence, the componentwise noise level $\sigma$ may be amplified (by factor of up to $\sqrt{n}$) when transforming to the basis in which the data are sparsest.
Perhaps surprisingly, the fact that the dictionary $D$ is unknown does not affect these rates; our lower bounds are both derived for fixed choices of $D$.
A similar phenomenon has been observed in IID SDL, where \citet{arora14newAlgorithms} showed upper bounds that are of the same rate as lower bounds for the case where the dictionary is known beforehand (the classic problem of \emph{sparse recovery}).

For $n \ll N$, an important consequence is that CSDL (and, similarly, IID SDL when $d \ll N$) is robust to arbitrary noise dependencies. This aspect of SDL is relatively unique amongst compressed sensing problems, and has important practical implications, explaining why these algorithms perform well for image, speech, or other highly dependent data.

On the negative side, our lower bounds imply that a high degree of sparsity is required to guarantee consistency.
Specifically, under sub-Gaussian noise, for fixed $n$ and $\sigma$, guaranteeing consistency requires $\lambda \in o \left( \frac{N}{\log N} \right)$. Our additional results in the Appendix suggest that an even higher degree (polynomially in $n$) of sparsity is required when the noise has tails that are heavier than Gaussian.

Our results focused on error measured in $\L_2$-loss, but they also have some consequences for rates under different losses. In particular, for $p \geq 2$, since the $\L_2$-norm on $\R^N$ dominates the $\L_p$-norm, our upper bounds extend to error measured in $\L_p$-loss. Similarly, for $p \in [1,2]$, since the $\L_p$-norm dominates the $\L_2$-norm, our lower bounds extend to error measured in $\L_p$-loss. On the other hand, further work is needed to obtain tight lower bounds for $p > 1$ and to obtain tight upper bounds for $p \in [1,2)$.

Our results rely on choosing the tuning parameter $\lambda$ well ($\lambda \approx \|R\|_{1,1}$). In the Appendix, we show that similar rates hold for the penalized form~\eqref{opt:penalized_CSDL} of CSDL, but that the tuning parameter $\lambda'$ should be chosen proportional to the noise level $\sigma$ (independent of $\|R\|_{1,1}$). Depending on the application and domain knowledge, it may be easier to estimate $\sigma$ or $\|R\|_{1,1}$, which may influence the choice of whether to use the constrained or penalized estimator (in conjunction with possible computational advantages of either form).

Finally, we reiterate that, while precise quantitative bounds are more difficult to study, our results extend beyond the TLGM; in general, the $\L_1$-constrained estimator converges to the optimal $\lambda$-sparse convolutional representation of $X$.

\subsection{Future Work}

There remain several natural open questions about the statistical properties of CSDL. 
First, how do rates extend to the case of multi-channel data $X \in \R^{N \times d}$? There are multiple possible extensions of CSDL to this case; the simplest is to make $R \in \R^{(N - n + 1) \times K \times d}$ and $D \in \R^{n \times K \times d}$ each $3$-tensors and learn separate dictionary and encoding matrices in each dimension, but another interesting approach may be to keep $R \in \R^{(N - n + 1) \times K}$ as a matrix and to make $D \in \R^{n \times K \times d}$ a $3$-tensor (and to generalize the multi-convolution operator appropriately), such that the positions encoded by $R$ are shared across dimensions, while different dictionary elements are learned in each dimension. Though a somewhat more restrictive model, this latter approach would have the advantage that statistical risk would \emph{decrease} with $d$, as data from multiple dimensions could contribute to the difficult problem of estimating $R$.

Another direction may be to consider a model with secondary spatial structure, such as correlations between occurrences of dictionary elements; for example, in speech data, consecutive words are likely to be highly dependent. This might be better modeled in a Bayesian framework, where $R$ is itself randomly generated with a certain (unknown) dependence structure between its columns.

Finally, while this work contributes to understanding the statistical properties of convolutional models, more work is needed to relate these results on sparse convolutional dictionaries to the hierarchical convolutional models that underlie state-of-the-art neural network methods for a variety of natural data, ranging from images to language and genomics\citep{krizhevsky12imagenet,lecun15CNNs,zhou15deepsea,singh16SPEID}. In this direction, \citet{papyan2017convolutional} very recently began making progress by extending the theoretical results of \citet{papyan17CSDLTheory} to multilayer convolutional neural networks.

\subsubsection*{Acknowledgements}
This material is based upon work supported by the NSF Graduate Research Fellowship DGE-1252522 (to S.S.).
The work was also partly supported by NSF IIS-1563887, Darpa D3M program, and AFRL (to B.P.), 
and NSF IIS-1717205 and NIH HG007352 (to J.M.).
Any opinions, findings, and conclusions or recommendations expressed in this material are those of the author(s) and do not necessarily reflect the views of the NSF and the NIH.

\bibliography{biblio}
\bibliographystyle{abbrvnat}

\include{appendix}
\end{document}

%% file: appendix.tex
\appendix

\section{Lemmas}

In this section, we collect, for easy reference, all lemmas that will be used in the proofs of our main results.

\begin{customlemma}{\ref{lemma:oracle_inequality}}[CSDL Oracle Bound]
Suppose $Y = X + \epsilon \in \R^N$.
Then, for any $\left( R, D \right) \in \S_\lambda$,
\begin{align}
\notag
& \| X - \hat R_\lambda \otimes \hat D_\lambda \|_2^2 \\
\label{ineq:oracle_inequality_appendix}
& \leq \| X -  R \otimes D \|_2^2 + 2\langle \epsilon, \hat R_\lambda \otimes \hat D_\lambda - R \otimes D \rangle.
\end{align}
\end{customlemma}
\begin{proof}
Let $(R,D) \in \S_\lambda$. Then, since $\left( \hat R_\lambda, \hat D_\lambda \right)$ minimizes the optimization problem~\eqref{opt:constrained_CSDL} for which $(R,D)$ is feasible,
\[\langle Y - \hat R_\lambda \otimes \hat D_\lambda, R \otimes D - \hat R_\lambda \otimes \hat D_\lambda \rangle
  \leq 0.\]
Rearranging this and using the fact that $Y = X + \epsilon$,
\[\langle X - \hat R_\lambda \otimes \hat D_\lambda, R \otimes D - \hat R_\lambda \otimes \hat D_\lambda \rangle
  \leq \langle \epsilon, \hat R_\lambda \otimes \hat D_\lambda - R \otimes D \rangle\]
Multiplying by $2$ and rewriting the left inner product by the polarization identity\footnote{Recall that, for any norm $\|\cdot\|$ induced by an inner product $\langle \cdot, \cdot \rangle$, $\|a\|^2 + \|b\|^2 - \|a - b\|^2 = 2\langle a, b \rangle$.} gives
\begin{align*}
& \| X - \hat R_\lambda \otimes \hat D_\lambda \|_2^2
  + \| R \otimes D - \hat R_\lambda \otimes \hat D_\lambda \|_2^2 \\
& \leq \| X -  R \otimes D \|_2^2 + 2\langle \epsilon, \hat R_\lambda \otimes \hat D_\lambda - R \otimes D \rangle,].
\end{align*}
Since $\| R \otimes D - \hat R_\lambda \otimes \hat D_\lambda \|_2^2 \geq 0$, this implies~\eqref{ineq:oracle_inequality_appendix}.
\end{proof}

We now note a well-known bound on the expected maximum of sub-Gaussian random variables:
\begin{lemma}[Sub-Gaussian Maximal Inequality]
Suppose an $\R^N$-valued random variable $\epsilon$ is componentwise sub-Gaussian with constant $\sigma$. Then,
\[\E \left[ \|\epsilon\|_\infty \right]
  \leq \sigma \sqrt{2\log(2N)}\]
and, for all $t > 0$,
\[\pr \left[ \|\epsilon\|_\infty > t \right]
  \leq 2 N \exp \left( -\frac{t^2}{2\sigma^2} \right).\]
Equivalently, for all $\delta > 0$,
\[\pr \left[ \|\epsilon\|_\infty > \sigma \sqrt{2 \log \left( \frac{2N}{\delta} \right)} \right]
  \leq \delta.\]
\label{lemma:subgaussian_maximal_inequality}
\end{lemma}

Proofs of these standard results can be found in, e.g., \citet{rigollet15HiDStatsNotes} (Theorem 1.14). Note that these bounds make no independence assumptions whatsoever on the coordinates of $\epsilon$. This will be crucial for both of our results under sub-Gaussian noise assumptions.

Second, we recall a classic inequality from analysis. The proof, which is based on clever use of H\"older's inequality, can found in \citet{beckner75inequalitiesFourier}.
\begin{lemma}[Young's Inequality for Convolutions]
Suppose $p,q,r \in [1,\infty]$ satisfy
\[1 + \frac{1}{r} = \frac{1}{p} + \frac{1}{q}.\]
Then, for any two $\R$-valued sequences $f$ and $g$,
\[\|f * g\|_r \leq \|f\|_p \|g\|_q.\]
\label{lemma:youngs}
\end{lemma}
A relevant corollary of Young's inequality is
\begin{corollary}[Bound on $\|R \otimes D\|_q$]
Consider two matrices $R \in \R^{(N - n + 1) \times K}$ and $D \in \R^{n \times K}$ such that, for some $q \geq 1$, $\|D\|_{q,\infty} \leq 1$. Then,
\[\|R \otimes D\|_q \leq \|R\|_{1,1} \|D\|_{q,\infty}.\]
\label{corr:trivial_bound}
\end{corollary}

\begin{proof}
By the triangle inequality and Young's inequality for convolutions,
\begin{align*}
\|R \otimes D\|_q
& = \left\| \sum_{j = 1}^K R_k * D_k \right\|_q \\
& \leq \sum_{j = 1}^K \| R_k * D_k \|_q \\
& \leq \sum_{j = 1}^K \|R_k\|_1 \| D_k \|_q \\
& \leq \|D\|_{q,\infty} \sum_{j = 1}^K \|R_k\|_1
  = \|R\|_{1,1} \|D\|_{q,\infty}
\end{align*}
\end{proof}

In particular, Corollary~\ref{corr:trivial_bound} implies that the average squared error of a trivial CSDL estimator that always estimates $R$ by $\hat R = 0 \in \R^{(N - n + 1) \times K}$ is always at most $\|R\|_{1,1}^2/N$.

Finally, our proofs will make use convolution matrices defined as follows:
\begin{definition}[Convolution Matrix]
Let $x \in \R^n$. Then, the matrix
\[T_{x,N}
  := \begin{bmatrix}
  x_1 & 0 & 0 & \cdots & 0 & 0 \\
  x_2 & x_1 & 0 & \cdots & 0 & 0 \\
  x_3 & x_2 & x_1 & \cdots & 0 & 0 \\
  \vdots & \vdots & \vdots & \ddots & \vdots & \vdots \\
  0 & 0 & 0 & \cdots & x_n & x_{n - 1} \\
  0 & 0 & 0 & \cdots & 0   & x_n \\
\end{bmatrix}
  \in \R^{N \times (N - n + 1)}\]
is called the length-$N$ convolution matrix of $x$.
\end{definition}
The convolution matrix is clearly named as such because, for $x \in \R^n$ and $y \in \R^{N - n + 1}$, $T_{x,N} y = x * y$. We will, in particular, make use of the fact that (transposes of) convolution matrices map jointly sub-Gaussian random variables to componentwise sub-Gaussian random variables, as expressed in the following lemma:
\begin{lemma}
Suppose $x \in \R^n$, and suppose an $\R^N$-valued random variable $\epsilon$ is jointly sub-Gaussian with constant $\sigma$. Then, the $\R^{N - n + 1}$-valued random variable $T_{x,N}^T \epsilon$ is componentwise sub-Gaussian with constant $\sigma \|x\|_2$.
\label{lemma:sub_gaussian_convolution}
\end{lemma}
\begin{proof}
Let $x \in \R^n$. We show this here for the $n^{th}$ coordinate $\left( T_{x,N}^T \epsilon \right)_n$ of $T_{x,N}^T \epsilon$. While the result clearly holds for other coordinates, more careful indexing is required for the first and last $n - 1$ coordinates, due to the structure of $T_x$. Since $\epsilon$ is jointly sub-Gaussian,
\begin{align*}
\E \left[ \exp \left( t \left( T_x^T \epsilon \right)_n \right) \right]
& = \E \left[ \exp \left( t\sum_{j = 1}^n x_j \epsilon_j \right) \right] \\
& \leq \exp \left( t^2 \|x\|_2^2 \sigma^2/2 \right).
\end{align*}
\end{proof}

Finally, our lower bounds are based on the following standard information-theoretic lower bound for estimating the mean of an $\L_1$-constrained Gaussian sequence:
\begin{lemma}
{\it (Corollary 5.16 of \citep{rigollet15HiDStatsNotes})} Consider the $\L_1$-constrained Gaussian sequence model, in which we observe $Z = \theta + \zeta \in \R^d$, where $\zeta \sim \mathcal{N}(0_d, \sigma^2 I_d)$ and we know that $\|\theta\|_1 \leq \lambda$. Then, we have the minimax lower bound
\[\inf_{\hat \theta} \sup_{\|\theta\|_1 \leq \lambda}
  \E \left[ \|\hat \theta - \theta\|_2^2 \right]
  \geq \frac{\lambda}{8} \min \left\{
          \lambda,
          \sigma \sqrt{\log d}
        \right\}\]
for estimating the model parameter $\theta$, where the infimum is taken over all estimators $\hat \theta$ of $\theta$ (i.e., all functions $\hat \theta : \R^d \to \R^d$). Moreover, this holds even if we know $\theta \in [0,\infty)^d$.
\label{lemma:Gaussian_sequence_model_lower_bound}
\end{lemma}
The $\min$ here reflects the fact that, in the extremely sparse or noisy regime $\lambda \leq \sigma \sqrt{\log d}$, the trivial estimator $\hat \theta = 0$ becomes optimal, with $\L_2$-risk at most $\lambda^2$. The last statement (that we can restrict to $\theta$ with non-negative entries), is not explicit in \citet{rigollet15HiDStatsNotes}, but can be easily seen from the proof of their Corollary 5.16, which involves restricting to $\theta$ on a non-negative multiple of the hypercube $\{0,1\}^d$.

\section{Proofs of Main Results}

This section provides proofs of the theorems mentioned in the main text.

\subsection{Upper Bounds for Constrained CSDL}
We now present proofs of our main upper bound results.

\begin{customthm}{\ref{thm:dependent_bound}}[Upper Bound for Componentwise Sub-Gaussian Noise]
Assume the TLGM holds, suppose the noise $\epsilon$ is componentwise sub-Gaussian with constant $\sigma$, and let the constrained CSDL tuning parameter $\lambda$ satisfy $\lambda \geq \|R\|_{1,1}$. Then, the reconstruction estimate $\hat X_\lambda = \hat R_\lambda \otimes \hat D_\lambda$ satisfies
\begin{equation}
\frac{1}{N} \E \left[ \|\hat X_\lambda - X\|_2^2 \right]
  \leq \frac{4 \lambda \sigma \sqrt{2 n \log(2N)}}{N}.
\label{ineq:dependent_bound_appendix}
\end{equation}
\end{customthm}

\begin{proof}
By the oracle bound (Lemma~\ref{lemma:oracle_inequality}) and the TLGM~\eqref{eq:TLGM_model},
\begin{equation}
\|\hat R_\lambda \otimes \hat D_\lambda - R \otimes D\|_2^2
  \leq 2 \langle \epsilon, \hat R_\lambda \otimes \hat D_\lambda - R \otimes D \rangle.
\label{ineq:basic_ineq}
\end{equation}
By H\"older's inequality,
\[\langle \epsilon, \hat R_\lambda \otimes \hat D_\lambda - R \otimes D \rangle \leq \|\epsilon\|_\infty \|\hat R_\lambda \otimes \hat D_\lambda - R \otimes D \|_1.\]
Then, by the triangle inequality and Young's inequality for convolutions (specifically Corollary~\ref{corr:trivial_bound}),
\begin{align*}
\|\hat R_\lambda \otimes \hat D_\lambda - R \otimes D \|_1
& \leq \|\hat R_\lambda \otimes \hat D_\lambda \|_1 + \| R \otimes D \|_1 \\
& \leq \|\hat R_\lambda\|_{1,1} \|\hat D_\lambda\|_{1,\infty} + \|R\|_{1,1} \|D\|_{1,\infty} \\
& \leq 2 \lambda \sqrt{n},
\end{align*}
where we used that fact that, for each $k \in [K]$, $\|\hat D_k\|_2 = \|D_k\|_2 = 1$ and $\hat D_k, D_k \in \R^n$, so that $\|\hat D_k\|_1, \|D_k\|_1 \leq \sqrt{n}$. Combining this series of inequalities with inequality~\eqref{ineq:basic_ineq} gives
\begin{align*}
\|\hat R_\lambda \otimes \hat D_\lambda - R \otimes D\|_2^2
\leq 4 \lambda \|\epsilon\|_\infty \sqrt{n}.
\end{align*}
Theorem~\ref{thm:dependent_bound} now follows by dividing by $N$ and applying the sub-Gaussian maximal inequality (Lemma~\ref{lemma:subgaussian_maximal_inequality}).
\end{proof}

\begin{customthm}{\ref{thm:independent_bound}}[Upper Bound for Jointly Sub-Gaussian Noise]
Assume the TLGM holds, suppose the noise $\epsilon$ is jointly sub-Gaussian with constant $\sigma$, and let the constrained CSDL tuning parameter $\lambda$ satisfy $\lambda \geq \|R\|_{1,1}$. Then, the reconstruction estimate $\hat X_\lambda = \hat R_\lambda \otimes \hat D_\lambda$ satisfies
\begin{equation}
\frac{1}{N} \E \left[ \|\hat X_\lambda - X\|_2^2 \right]
  \leq \frac{4 \lambda \sigma \sqrt{2 \log(2(N - n + 1))}}{N}.
\label{ineq:independent_bound_appendix}
\end{equation}
\end{customthm}
\begin{proof}
By Lemma~\ref{lemma:sub_gaussian_convolution}, for each $k \in [K]$, $T_{D_k}^T \epsilon$ is componentwise sub-Gaussian with constant $\sigma \|D_k\|_2 = \sigma$. Thus, the sub-Gaussian maximal inequality (Lemma~\ref{lemma:subgaussian_maximal_inequality}) implies
\[\E \left[ \left\| T_{D_k}^T \epsilon \right\|_\infty \right] \leq \sigma \sqrt{2 \log(2(N - n + 1))},\]
and so, by H\"older's inequality,
\begin{align*}
\E \left[ \left| \langle \epsilon, R_k * D_k \rangle \right| \right]
& = \E \left[ \left| \langle T_{D_k}^T \epsilon, R_k \rangle \right| \right] \\
& \leq \E \left[ \|T_{D_k}^T \epsilon\|_\infty \|R_k\|_1 \right] \\
& \leq \lambda \sigma \sqrt{2 \log(2(N - n + 1))}.
\end{align*}
Similarly, we can bound $\E \left[ \left| \langle \epsilon, (\hat R_\lambda)_k * (\hat D_\lambda)_k \rangle \right| \right]$ by $\lambda \sigma \sqrt{2 \log(2(N - n + 1))}$.
As in the componentwise sub-Gaussian case, the remainder of the proof consists of applying the oracle bound~\eqref{ineq:oracle_inequality} and the triangle inequality.
\end{proof}

\subsection{Lower Bounds}
We now present proofs of our main lower bound results.

\begin{customthm}{\ref{thm:joint_lower_bound}}[Lower Bound for Jointly sub-Gaussian Noise]
Assume the TLGM holds, and suppose that $\epsilon \sim \mathcal{N}(0, \sigma^2 I_N)$, so that $\epsilon$ is jointly sub-Gaussian with constant $\sigma$. Then, the following lower bound on the minimax average $\L_2$ reconstruction risk holds:
\begin{equation}
M(\lambda, N, n, \sigma) \geq \frac{\lambda}{8N} \min \left\{
         \lambda,
         \sigma \sqrt{\log(N - n + 1)}
       \right\}
\label{ineq:joint_lower_bound_appendix}
\end{equation}
\end{customthm}

\begin{proof}
To prove a lower bound, we can fix the dictionary $D$; doing so can only decrease the supremum in the definition of $M(\lambda, N, n, \sigma)$. In particular, let
\[D = [1,0,0,...,0]^T \in \R^n\]
denote the first canonical basis vector of $\R^n$.

Let $\I := T_D \left( [0,\infty)^{N - n + 1} \right) \subseteq \R^N$ denote the image of $[0,\infty)^{N - n + 1}$ under $T_D$. Noting that $\I$ is a convex set, let $\Pi_\I : \R^N \to \I$ denote the $\L_2$-projection operator onto $\I$; i.e.,
\[\Pi_\I(x)
  = \argmin_{y \in \I} \|y - x\|_2,
  \quad \forall x \in \R^N.\]
Also, it is easy to check that $T_D$ has a left inverse $T_D\inv : \I \to \R^{N - n + 1}$ such that, $T_D\inv T_D = I_{N - n + 1}$ (in fact, $T_D\inv = T_D^T$), and that, for all $x \in \I$,
\begin{equation}
\|T_D\inv x\|_2 = \|x\|_2.
\label{ineq:T_D_inv_norm1}
\end{equation}
Suppose we have an estimator $\hat X$ of $R \otimes D$ given $R \otimes D + \epsilon$ (so that $\hat X$ is a function from $\R^N$ to $\R^N$). Then, given an observation $Z = R + \zeta \in \R^{N - n + 1}$, where $\zeta \sim \mathcal{N}(0, \sigma^2 I)$, define the estimator
\[\hat R = T_D\inv \left( \Pi\left( \hat X(Z \otimes D) \right) \right)\]
of $R$. Then, by inequality~\eqref{ineq:T_D_inv_norm1} and the fact that $T_D R \in \I$,
\begin{align*}
\left\| \hat R - R \right\|_2
& = \left\| \hat T_D\inv \left( \Pi_\I \left( \hat X((R + \zeta) \otimes D) \right) - T_D R \right) \right\|_2 \\
& = \left\| \Pi_\I \left( \hat X((R + \zeta) \otimes D) \right) - T_D R \right\|_2 \\
& \leq \left\| \hat X((R + \zeta) \otimes D) - T_D R \right\|_2 \\
& = \left\| \hat X(R \otimes D + \zeta \otimes D) - R \otimes D \right\|_2
\end{align*}

It is trivial to check that $\epsilon = T_D \zeta$ is jointly sub-Gaussian with constant $\sigma$. Thus, after taking an infimum over $R$ with $\|R\|_{1,1} \leq \lambda$ and a supremum over estimators $\hat X$ on both sides, the lower bound follows from the $\L_1$-constrained Gaussian sequence lower bound (Lemma \ref{lemma:Gaussian_sequence_model_lower_bound}).

\end{proof}
\begin{customthm}{\ref{thm:componentwise_lower_bound}}[Lower Bound for Componentwise sub-Gaussian Noise]
Assume the TLGM holds. Then, there exists a (Gaussian) noise pattern $\epsilon$ that is componentwise sub-Gaussian with constant $\sigma$ such that the following lower bound on the minimax average $\L_2$ reconstruction risk holds:
\begin{equation}
M(\lambda, N, n, \sigma) \geq \frac{\lambda}{8N} \min \left\{
         \lambda,
         \sigma \sqrt{n \log(N - n + 1)}
       \right\}
\label{ineq:componentwise_lower_bound_appendix}
\end{equation}
\end{customthm}

\begin{proof}
The proof is similar to the jointly sub-Gaussian case, but with a different (fixed) choice of dictionary $D$. In particular, let
\[D = \left[ \frac{1}{\sqrt{n}}, \frac{1}{\sqrt{n}}, ..., \frac{1}{\sqrt{n}} \right]^T \in \R^n\]
denote the non-negative uniform $\L_2$-unit vector in $\R^n$. As previously, let $\I := T_D \left( [0,\infty)^{N - n + 1} \right) \subseteq \R^N$ denote the image of $[0,\infty)^{N - n + 1}$ under $T_D$. $\I$ is still convex, and so we can let $\Pi_\I : \R^N \to \I$ denote the $\L_2$-projection operator onto $\I$; i.e.,
\[\Pi_\I(x)
  = \argmin_{y \in \I} \|y - x\|_2,
  \quad \forall x \in \R^N.\]
Also, it is easy to check that $T_D$ has full rank $N - n + 1$, and therefore has a left inverse $T_D\inv : \I \to \R^{N - n + 1}$ such that, $T_D\inv T_D = I_{N - n + 1}$.
Moreover, for all $y \in [0,\infty)^{N - n + 1}$, $\|T_Dy\|_2 \geq \|y\|_2,$ so that, for all $x \in \I$,
\begin{equation}
\|T_D\inv x\|_2 \leq \|x\|_2.
\label{ineq:T_D_inv_norm2}
\end{equation}
Suppose we have an estimator $\hat X$ of $R \otimes D$ given $R \otimes D + \epsilon$ (so that $\hat X$ is a function from $\R^N$ to $\R^N$). Then, given an observation $Z = R + \zeta \in \R^{N - n + 1}$, where $\zeta \sim \mathcal{N}(0, n\sigma^2 I)$, define the estimator
\[\hat R = T_D\inv \left( \Pi\left( \hat X(Z \otimes D) \right) \right)\]
of $R$. Then, by inequality~\eqref{ineq:T_D_inv_norm2} and the fact that $T_D R \in \I$,
\begin{align*}
\left\| \hat R - R \right\|_2
& = \left\| \hat T_D\inv \left( \Pi_\I \left( \hat X((R + \zeta) \otimes D) \right) - T_D R \right) \right\|_2 \\
& = \left\| \Pi_\I \left( \hat X((R + \zeta) \otimes D) \right) - T_D R \right\|_2 \\
& \leq \left\| \hat X((R + \zeta) \otimes D) - T_D R \right\|_2 \\
& = \left\| \hat X(R \otimes D + \zeta \otimes D) - R \otimes D \right\|_2
\end{align*}

It remains only to observe that $\epsilon = T_D \zeta$ is componentwise sub-Gaussian with constant $\sigma$ (although it is only \emph{jointly} sub-Gaussian with constant $\sqrt{n} \sigma$). Thus, after taking an infimum over $R$ with $\|R\|_{1,1} \leq \lambda$ and a supremum over estimators $\hat X$ on both sides, the lower bound follows from the $\L_1$-constrained Gaussian sequence lower bound (Lemma \ref{lemma:Gaussian_sequence_model_lower_bound}).
\end{proof}

\subsection{Comparison to IID SDL}
Here, we prove the upper bound for IID SDL in the LGM setting, which we used to compare the performance of IID SDL and CSDL in Section~\ref{sec:IID_SDL_comparison}. For notational simplicity, we drop the convention (used in that section) of using the prime symbol $'$ to denote quantities from the LGM; in this section, all quantities are as in the LGM.
\begin{customthm}{\ref{thm:IID_SDL_bound}}[Upper Bound for IID SDL]
Assume the LGM holds, suppose the noise $\epsilon$ is componentwise sub-Gaussian with constant $\sigma$ (more precisely, for each dimension $j \in [d']$, $\epsilon_j \in \R^{N'}$ is componentwise sub-Gaussian with constant $\sigma$), and let the constrained IID SDL parameter $\lambda'$ satisfy $\lambda' \geq \|R'\|_{1,1}$. Then, the reconstruction estimate $\hat X_{\lambda'}' = \hat R_{\lambda'}' \hat D_{\lambda'}'$ satisfies
\begin{equation}
\frac{1}{Nd} \|X - \hat X_\lambda\|_{2,2}^2
  \leq \frac{4 \lambda' \sigma \sqrt{2d \log(2Nd)}}{Nd}
\label{ineq:IID_SDL_bound_appendix}
\end{equation}
\end{customthm}

\begin{proof}
Note that, by the triangle inequality and the fact that each $\|D_k\|_1 \leq \sqrt{d}$ (since $\|D_k\|_2 \leq 1$ and $D_k \in \R^d$),
\begin{align*}
\|RD\|_{1,1}
& = \sum_{i = 1}^N \|(RD)_i\|_1 \\
& = \sum_{i = 1}^N \left\| \sum_{k = 1}^K R_{i,k} D_k \right\|_1 \\
& \leq \sum_{i = 1}^N \sum_{k = 1}^K \left\| R_{i,k} D_k \right\|_1 \\
& = \sum_{i = 1}^N \sum_{k = 1}^K |R_{i,k}| \| D_k \|_1 \\
& \leq \sqrt{d} \sum_{i = 1}^N \sum_{k = 1}^K |R_{i,k}|
  = \sqrt{d} \|R\|_{1,1} \leq \sqrt{d} \lambda.
\end{align*}
Similarly, $\|\hat R_\lambda \hat D_\lambda\|_1 \leq \sqrt{d} \lambda$.
By the same proof as the CSDL oracle inequality~\eqref{ineq:oracle_inequality}, one can show that, under the linear generative model~\ref{eq:LGM},
\[\|RD - \hat R_\lambda \hat D_\lambda\|_{2,2}^2
  \leq 2\langle \epsilon, \hat R_\lambda \hat D_\lambda - RD \rangle_F,\]
where $\langle \cdot, \cdot \rangle_F$ denotes the Frobenius inner product. Thus, by H\"older's inequality and the sub-Gaussian maximal inequality (Lemma~\ref{lemma:subgaussian_maximal_inequality})
\begin{align*}
\|RD - \hat R_\lambda \hat D_\lambda\|_{2,2}^2
& \leq 2 \langle \epsilon, \hat R_\lambda \hat D_\lambda - RD \rangle_F \\
& \leq 2 \|\epsilon\|_{\infty,\infty} \| \hat R_\lambda \hat D_\lambda - RD \|_{1,1} \\
& \leq 2 \sigma \sqrt{2 \log(2Nd)} \left( \hat R_\lambda \hat D_\lambda \|_{1,1} + \| RD \|_{1,1} \right) \\
& \leq 4 \lambda \sigma \sqrt{2d \log(2N)}
\end{align*}
\end{proof}

\section{Additional Theoretical Results}
\label{app:extra_theory}
In this section, we provide proofs of a few additional upper bound results for CSDL that complement our main results. Specifically, we consider a milder finite-moment noise condition, and also consider the penalized form~\eqref{opt:penalized_CSDL} of CSDL.

\subsection{Upper Bounds under Moment Conditions}
First, while our main results consider variants of sub-Gaussian noise, we here also consider a variant with heavy-tailed noise, where we assume only that the (arbitrarily dependent) noise has some number of finite moments. The resulting guarantee holds under almost trivially weak assumptions, but the bound is also quite weak, requiring an extremely high degree of sparsity to guarantee consistency.
Experiment~\ref{experiment:4} in Appendix~\ref{app:extra_experiments} also provides experimental evidence that the rate of the upper bound may be tight, although minimax lower bounds are needed to be certain.
The precise result is as follows:
\begin{theorem}[Upper Bound under Noise with Componentwise Finite Moments]
Assume the TLGM holds, suppose that, for some $p \in [1, \infty]$, the components of the noise $\epsilon$ have finite $p^{th}$ moment at most $\mu_p$; that is,
\[\mu_p := \sup_{i \in [N]} \left( \E \left[ \epsilon_i^p \right] \right)^{1/p} < \infty.\]
Let the constrained CSDL tuning parameter $\lambda$ satisfy $\lambda \geq \|R\|_{1,1}$. Then, the reconstruction estimate $\hat X_\lambda = \hat R_\lambda \otimes \hat D_\lambda$ satisfies
\begin{equation}
\frac{1}{N} \E \left[ \|\hat X_\lambda - X\|_2^2 \right]
  \leq 4 \lambda \mu_p N^{\frac{1 - p}{p}} n^{\max\left\{ 0,\frac{p - 2}{2p} \right\}}.
\label{ineq:moment_bound}
\end{equation}
\label{thm:moment_bound_rate}
\end{theorem}

\begin{proof}
As in the sub-Gaussian case, the proof begins with the oracle inequality~\eqref{ineq:oracle_inequality}, which, under the TLGM, gives
\[\|D \otimes R - \hat D_\lambda \otimes \hat R_\lambda\|_2^2
  \leq 2\langle \epsilon, \hat R_\lambda \otimes \hat D_\lambda - R \otimes D \rangle.\]
By H\"older's inequality,
\[\langle \epsilon, \hat R_\lambda \otimes \hat D_\lambda - R \otimes D \rangle \leq \|\epsilon\|_p \|\hat R_\Lambda \otimes \hat D_\lambda - R \otimes D \|_q,\]
where $q = \frac{p}{p - 1} \geq 1$.

If $q \geq 2$ (i.e., if $p \leq 2$), then, for each $k \in [K]$,
$\|D_k\|_q \leq \|D_k\|_2 = 1$ and $\|(\hat D_\lambda)_k\|_q \leq \|(\hat D_\lambda)_k\|_2 = 1$. Otherwise
\[\|D_k\|_q \leq n^{1/q - 1/2} \|D_k\|_2 = n^{1/q - 1/2} = n^{\frac{p - 2}{2p}},\] and, similarly, $\|(\hat D_\lambda)_k\|_q \leq n^{\frac{p - 2}{2p}}$. In short,
\[\|D\|_{q,\infty}, \|\hat D_\lambda\|_{q,\infty}
  \leq n^{\max\left\{ 0,\frac{p - 2}{2p} \right\}}.\]

Hence, by the triangle inequality and Corollary~\ref{corr:trivial_bound} of Young's inequality for convolutions,
\begin{align*}
\|\hat R_\lambda \otimes \hat D_\lambda - R \otimes D \|_q
& \leq \|\hat R_\lambda \otimes \hat D_\lambda \|_q + \| R \otimes D \|_q \\
& \leq \|\hat R_\lambda\|_{1,1} \|\hat D_\lambda\|_{q,\infty} + \|R\|_{1,1} \|D\|_{q,\infty} \\
& \leq 2\lambda n^{\max\left\{ 0,\frac{p - 2}{2p} \right\}},
\end{align*}

Combining these inequalities, we have
\begin{align*}
\|\hat R_\lambda \otimes \hat D_\lambda - R \otimes D\|_2^2
  \leq 4 \lambda \|\epsilon\|_2 n^{\max\left\{ 0,\frac{p - 2}{2p} \right\}}.
\end{align*}
Theorem~\ref{thm:moment_bound_rate} now follows by observing that
\[\E \left[ \|\epsilon\|_p \right]
  \leq \left( \E \left[ \|\epsilon\|_p^p \right] \right)^{1/p}
  \leq \mu_p N^{1/p}.
\]
\end{proof}

\subsection{Upper Bounds for Penalized CSDL}
In this section, we show that the upper bounds presented previously for $\L_1$-constrained CSDL~\eqref{opt:constrained_CSDL} also hold for $\L_1$-penalized CSDL~\eqref{opt:penalized_CSDL}. The only major difference between the constrained and penalized forms (aside from computational considerations) is that the tuning parameters $\lambda$ (for the constrained form) and $\lambda'$ (for the penalized form) should be chosen differently; for the constrained form, $\lambda = \|R\|_{1,1}$ is optimal, whereas, for the penalized form, $\lambda' = \sigma \sqrt{2 \log(2N)}$ is optimal. This difference can be practically significant, in that either $\|R\|_{1,1}$ or $\sigma$ may be easier to estimate.

We now demonstrate the bound for the case of componentwise sub-Gaussian noise; the proofs for jointly sub-Gaussian or finite-moment noise can then be easily derived by analogy with the proofs for constrained CSDL.

\begin{theorem}[Upper Bound for Penalized CSDL under Componentwise Sub-Gaussian Noise]
Let $\delta \in (0,1)$. Assume the TLGM holds, suppose the noise $\epsilon$ is componentwise sub-Gaussian with constant $\sigma$, and let the penalized CSDL tuning parameter $\lambda'$ satisfy $\lambda' = \sigma \sqrt{2 \log \left( \frac{2N}{\delta} \right)}$. Then, with probability at least $1 - \delta$, the reconstruction estimate $\hat X_{\lambda'} = \hat R_{\lambda'} \otimes \hat D_{\lambda'}$ based on penalized CSDL~\eqref{opt:penalized_CSDL} satisfies
\begin{equation}
\frac{1}{N} \|\hat R_{\lambda'} \otimes \hat D_{\lambda'} - R \otimes D\|_2^2
  \leq \frac{4 \lambda' \sigma \sqrt{2 n \log(2N/\delta)}}{N}.
\label{ineq:penalized_CSDL_bound}
\end{equation}
\label{thm:penalized_CSDL_bound}
\end{theorem}

\begin{proof}
By construction of $\left( \hat R_{\lambda'}, \hat D_{\lambda'} \right)$,
\[\|\hat R_{\lambda'} \otimes \hat D_{\lambda'} - Y\|_2^2 + \lambda' \|\hat R_{\lambda'}\|_{1,1}
  \leq \|R \otimes D - Y\|_2^2 + \lambda' \|R\|_{1,1}.\]
Expanding $Y = X + \epsilon = R \otimes D + \epsilon$ and rearranging gives
\begin{align*}
\left\| X - \hat X_{\lambda'} \right\|_2^2
& \leq 2 \langle \epsilon, \hat R_{\lambda'} \otimes \hat D_{\lambda'} - R \otimes D \rangle \\
& + \lambda' \left( \|R\|_{1,1} - \|\hat R_{\lambda'}\|_{1,1} \right)
\end{align*}
Again, by H\"older's inequality
\[\langle \epsilon, \hat R_{\lambda'} \otimes \hat D_{\lambda'} - R \otimes D \rangle
  \leq \|\epsilon\|_\infty \|R \otimes D - \hat R_{\lambda'} \otimes \hat D_{\lambda'}\|_1,\]
and, by the tail-bound form of the sub-Gaussian maximal inequality (Lemma~\ref{lemma:subgaussian_maximal_inequality}) with probability at least $1 - \delta$,
\[\|\epsilon\|_\infty \leq \sigma \sqrt{2 \log \left( \frac{2N}{\delta} \right)}.\]
By the triangle inequality and Young's inequality (Lemma~\ref{corr:trivial_bound})
\begin{align*}
\|R \otimes D - \hat R_{\lambda'} \otimes \hat D_{\lambda'}\|_1
& \leq \|R\|_{1,1} \|D\|_{1,\infty} - \|\hat R_{\lambda'}\|_{1,1} \|\hat D_{\lambda'}\|_{1,\infty} \\
& \leq \sqrt{n} \left( \|R\|_{1,1} + \|\hat R_{\lambda'}\|_{1,1} \right)
\end{align*}
Hence, with probability at least $1 - \delta$, if $\lambda' \geq \sigma \sqrt{2 n \log \left( 2N/\delta \right)}$,
\begin{align*}
\left\| R \otimes D - \hat R_{\lambda'} \otimes \hat D_{\lambda'} \right\|_2^2
& \leq 2 \sigma \sqrt{2 n \log \left( 2N/\delta \right)} \left( \|R\|_{1,1} + \|\hat R_{\lambda'}\|_{1,1} \right) \\
& + 2 \lambda' \left( \|R\|_{1,1} - \|\hat R_{\lambda'}\|_{1,1} \right)
  \leq 2\lambda' \|R\|_{1,1}.
\end{align*}
For the specific value $\lambda' = \sigma \sqrt{2 n \log \left( 2N/\delta \right)}$, this implies
\[\frac{1}{N} \left\| \hat R_{\lambda'} \otimes \hat D_{\lambda'} - R \otimes D \right\|_2^2
  \leq 2 \sigma \lambda' \frac{\sqrt{2 n \log \left( 2N/\delta \right)}}{N}.\]
\end{proof}

\section{Details of Optimization Algorithm and Experiments}

The constrained CSDL estimator~\eqref{opt:constrained_CSDL} was computed using a simple alternating projected gradient descent algorithm, which iteratively performs the following four steps: 1) gradient step with respect to $D$, 2) project the columns of $D$ onto the unit sphere, 3) gradient step with respect to $R$, and 4) project $R$ (with respect to Frobenius norm) into the intersection of the non-negative orthant and the $\L_{1,1}$ ball of radius $\lambda$.
This algorithm was run for $200$ iterations, with decaying step size $0.01 i^{-1/2}$, where $i \in [200]$ is the iteration number.
Pseudocode is given in Algorithm~\ref{alg:alternating_minimization}.
In the pseudocode, $\text{Project}(x, p, r)$ denotes a subroutine that returns the projection (with respect to $\L_2$/Frobenius norm) of $x$ onto the (convex) $p$-norm ball of radius $r$.

\begin{algorithm}
\begin{algorithmic}
\State $\hat D_\lambda \gets \text{Project}(\mathcal{N}(0_{n \times K}), I, (2, \infty), 1)$
\State $\hat R_\lambda \gets \text{Project}(\mathcal{N}(0_{(N - n + 1) \times K}, I), (1, 1), \lambda)$
\For{$i \gets 1;\; i <= 200;\; i++$}
\State $\gamma \gets 0.01i^{-1/2}$
\State $\hat D_\lambda \gets \hat D_\lambda - \gamma \nabla_D \|X - \hat R_\lambda \otimes \hat D_\lambda\|_2^2$
\State $\hat D_\lambda \gets \text{Project}(\hat D_\lambda, (2, \infty), 1)$
\State $\hat R_\lambda \gets \hat R_\lambda - \gamma \nabla_R \|X - \hat R_\lambda \otimes \hat D_\lambda\|_2^2$
\State $\hat R_\lambda \gets \text{Project}(\hat R_\lambda, (1, 1), \lambda)$
\EndFor
\State $\hat X_\lambda \gets \hat R_\lambda \otimes \hat D_\lambda$
\end{algorithmic}
\caption{Alternating minimization used to optimize~\eqref{opt:constrained_CSDL}.}
\label{alg:alternating_minimization}
\end{algorithm}

{\it Parameters of Experimental Setup:}
In all experiments, unless noted otherwise, the data are generated using the following parameter settings:
\begin{itemize}
\setlength{\itemsep}{0pt}
\item
sequence length $N = 1000$
\item
$\L_1$-Sparsity $\|R\|_{1,1} = 100$
\item
Dictionary element length $n = 10$
\item
Dictionary size $K = 5$
\item
Noise level $\sigma = 0.1$
\end{itemize}

\section{Additional Experimental Results}
\label{app:extra_experiments}
Here, we present additional experiments on simulated data, that further support our theoretical results.

\begin{experiment}
\label{experiment:3}
Our third experiment studies the sensitivity of the estimator $\hat X_\lambda$ to its tuning parameter. Figure~\ref{fig:exp3} shows error as a function of $\lambda$ for logarithmically spaced valued between $10^{-2}$ and $10^4$. The error appears robust to setting $\lambda \gg \|R\|_{1,1}$ (although the upper bound becomes quite loose). In fact, the error does not appear to increase for $\lambda \geq \|R\|_{1,1}$, so it appears that, in this regime, the estimator may be adaptive to the true $\|R\|_{1,1}$. More work is needed to determine if this is the case in general.

\begin{figure}
\centering
\includegraphics[width=\linewidth,trim={0mm 0mm 0mm 0mm},clip]{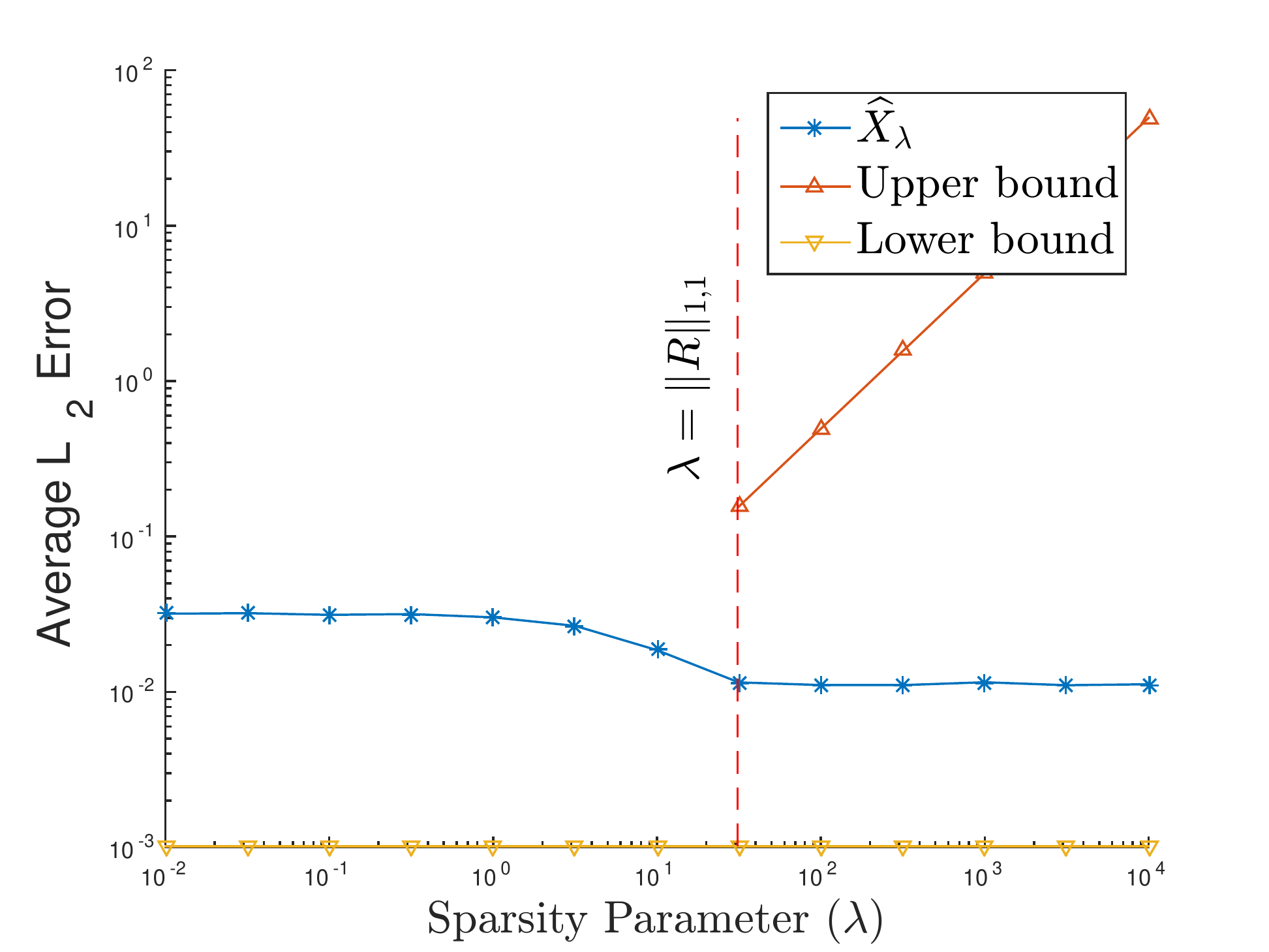}
\caption{Experiment 3: Average $\L_2$-error as a function of the tuning parameter $\lambda$ of $\hat X_\lambda$. The dashed line indicates the $\L_1$-sparsity $\|R\|_{1,1} = \left\lfloor \sqrt{N} \right\rfloor = 33$. Note that the upper bound only applies when $\lambda \geq \|R\|_{1,1}$.}
\label{fig:exp3}
\end{figure}
\end{experiment}

\begin{experiment}
\label{experiment:4}
This experiment mimics Experiment~\ref{experiment:1} from the main paper, but uses heavy tailed noise with only finitely many finite moments. Specifically, we sample the entries of $\epsilon$ IID from a symmetric generalized Pareto distribution with threshold (location) parameter $\theta = 2$, scale parameter $\sigma = 1$, and tail index (shape) parameter $\xi = 1/2$, which has probability density function
\[\frac{1}{2\sigma} \left( 1 + \frac{\xi (|x| - \mu)}{\sigma} \right)^{-\frac{1}{\xi} - 1}\]
supported on $(-\infty,-2) \cup (2,\infty)$. This choice of $\epsilon$ has $\mu_p = \left( \E \left[ \epsilon_i^p \right] \right)^{1/p} < \infty$ if and only if $p < 2$, and so Theorem~\ref{thm:moment_bound_rate} suggests we may see a slower convergence rate of order $\|R\|_{1,1}N^{-1/2}$. Note that, in this case, the trivial estimator $\hat X_\infty$ has infinite $\L_2$ error, and so we excluded it from the simulation.
Figure~\ref{fig:exp4} shows error as a function of $N$ for logarithmically spaced values between $10^2$ and $10^4$, with $\|R\|_{1,1}$ scaling as constant $\|R\|_{1,1} = 5$, square-root $\|R\|_{1,1} = \left\lfloor \sqrt{N} \right\rfloor$, and linearly $\|R\|_{1,1} = \lfloor N/10 \rfloor$.
The results appear consistent with the main prediction of Theorems~\ref{thm:moment_bound_rate}, namely that, using the optimal tuning parameter $\lambda = \|R\|_{1,1}$, the CSDL estimator is consistent only in the case where $\|R\|_{1,1} \in o \left( \sqrt{N} \right)$.
Note that, in these highly noisy settings, the trivial estimator $\hat X_0 = 0$ tends to perform best in terms of average $\L_2$ error, suggesting that little or no meaningful information can be extracted from $X$.
However, for many applications, average $\L_2$ error may not be the best performance measure, and it is possible that CSDL may still extract some useful information.

\begin{figure}
\centering
\includegraphics[width=\linewidth,trim={7mm 0mm 17mm 0mm},clip]{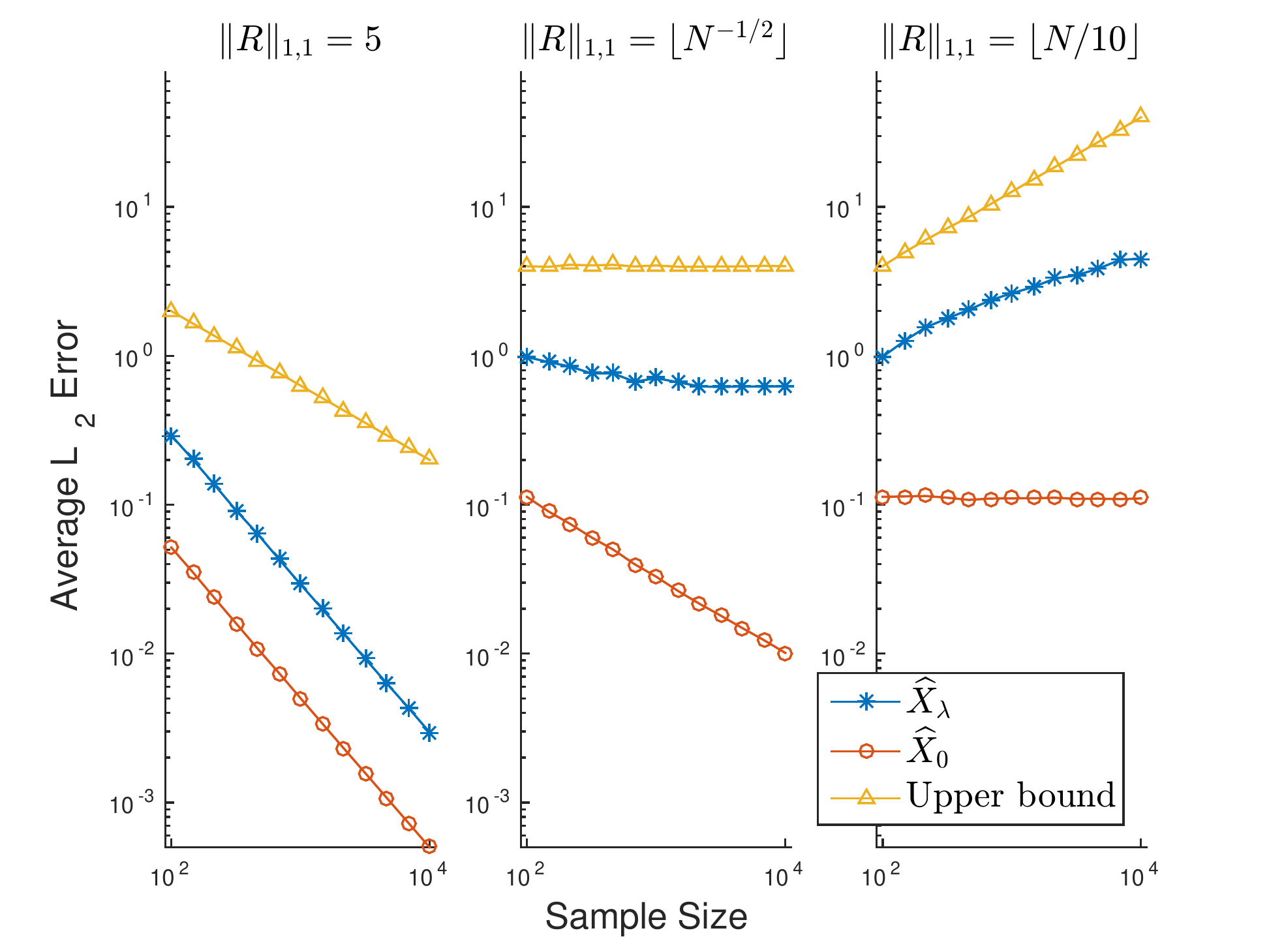}
\caption{Experiment 4: Average $\L_2$-error as a function of sequence length $N$, in the case of heavy tailed noise, with sparsity scaling as $\|R\|_{1,1} = 5$ (first panel), $\|R\|_{1,1} = \left\lfloor \sqrt{N} \right\rfloor$ (second panel), and $\|R\|_{1,1} = \lfloor N/10 \rfloor$ (third panel).}
\label{fig:exp4}
\end{figure}
\end{experiment}